\documentclass[preprint,pdftex]{imsart}
\usepackage{amsthm,amsmath,amssymb,epic,latexsym,multirow,ulem}

\usepackage[numbers]{natbib}
\RequirePackage{natbib}
\RequirePackage[colorlinks,citecolor=blue,urlcolor=blue]{hyperref}

\usepackage{wrapfig}
\usepackage[english]{babel}
 \usepackage[T1]{fontenc} 
 \usepackage[ddmmyyyy,hhmmss]{datetime}

\usepackage[svgnames,dvipsnames]{xcolor}
 \usepackage{makeidx}
 \usepackage{fancyhdr}
\usepackage{epsf}
\usepackage[dvips]{epsfig}
\usepackage{subfig}
 \usepackage{graphicx}
\usepackage[nice]{nicefrac}
\usepackage{dsfont}

\usepackage{wasysym}
\usepackage{stmaryrd}
\usepackage{aeguill}
\usepackage{mathrsfs}
\usepackage{enumerate}
\usepackage{enumitem}

\bigskip

\theoremstyle{plain}
\newtheorem{theo}{Theorem}[section]
\newtheorem*{theo*}{Theorem}
\newtheorem{cor}[theo]{Corollary}
\newtheorem{prop}[theo]{Proposition}
\newtheorem{lem}[theo]{Lemma}

\newtheorem*{defi*}{Definition}
\newtheorem{rem}[theo]{Remark}

\newcommand{\N}{\mathbb{N}}

\newcommand{\Z}{\mathbb{Z}}
\newcommand{\E}{\mathbb{E}}
\renewcommand{\P}{\mathbb{P}}

\newcommand{\lo}{\mathscr{L}}

\newcommand{\V}{{\mathbb V}\mathrm{ar}}

\newcommand{\red}[1]{{\color{black}{#1}}} 

\newcommand{\un}{\mathds{1}}

\newcommand{\ox}{\overline{x}} 
\newcommand{\oX}{\overline{X}} 
\newcommand{\oz}{\overline{z}} 
\newcommand{\oy}{\overline{y}} 
\newcommand{\ou}{\overline{u}} 
\newcommand{\Hm}{\mathscr{H}}

\newcommand{\co}{ {\mathbf{c_{_0}}}}
\newcommand{\cn}{ {\mathbf{c}}}

\newcommand{\cun}{ {\mathbf{c_{_1}}}}
\newcommand{\cde}{ {\mathbf{c_{_2}}}}

\begin{document}

\title{Simple random walk on $\Z^2$ perturbed on the axis (renewal case)}



\author{\fnms{Pierre} \snm{Andreoletti}\ead[label=e1]{Pierre.Andreoletti@univ-orleans.fr}}
\address{
Institut Denis Poisson,   UMR C.N.R.S. 7013,
Universit\'e d'Orl\'eans, Orl\'eans, France. \printead{e1}}

\and 
\author{\fnms{Pierre} \snm{Debs}\ead[label=e2]{Pierre.Debs@univ-orleans.fr}}
\address{Institut Denis Poisson,   UMR C.N.R.S. 7013,
Universit\'e d'Orl\'eans, Orl\'eans, France. \printead{e2}} 
\runauthor{Andreoletti, Debs}

\begin{abstract}
We study a simple random walk on $\Z^2$ with constraints on the axis. Motivation comes from physics when particles (a gas for example, see \cite{Dalibard}) are submitted to a local field. In our case we assume that the particle evolves freely in the cones but when touching the axis a force pushes it back progressively to the origin. The main result proves that this force can be parametrized in such a way that a renewal structure appears in the trajectory of the random walk. This  implies the existence of an ergodic result for the parts of the trajectory restricted to the axis.

\end{abstract}

\maketitle

\section{Introduction}

We consider  a random walk $\mathbf{X}=(X_n,{n\in\mathbb N})$ on $\mathbb Z^2$ starting at $(1,1)$, which probabilities of transition differ whether the walk is on the axis, denoted by $K^c:=\lbrace (x,y)\in\mathbb Z^2, {xy=0}\rbrace$, or on $K$ which is made up of four cones. More precisely, $\mathbf X$ is a simple random walk on $K$, that is
$p(x, x\pm e_i)=\nicefrac{1}{4}$, for all $x\in K$, where $e_i$ is a vector of the canonical basis, whereas on $K^c\backslash\lbrace(0,0)\rbrace$ the walk is pushed towards the origin: there exists $\alpha \geq 0$ such that for all    $i>0$ 
\begin{align*}
p((i,0),(i+1,0))&=p((i,0),(i, \pm 1))=p((0,i),(0,i+1))=p((0,i),(\pm 1, i))=\frac{1}{4i^{\alpha}},\\
p((i,0),(i-1,0))&=p((0,i),(0,i-1))=1-\frac{3}{4i^{\alpha}},
\end{align*}
and symmetrically when $i<0$ (see figure \ref{whynot}). Note that the origin is a ``special'' point as $p((0,0),\pm e_i)=\nicefrac{1}{4}$. 
\begin{figure}
\includegraphics[scale=0.7]{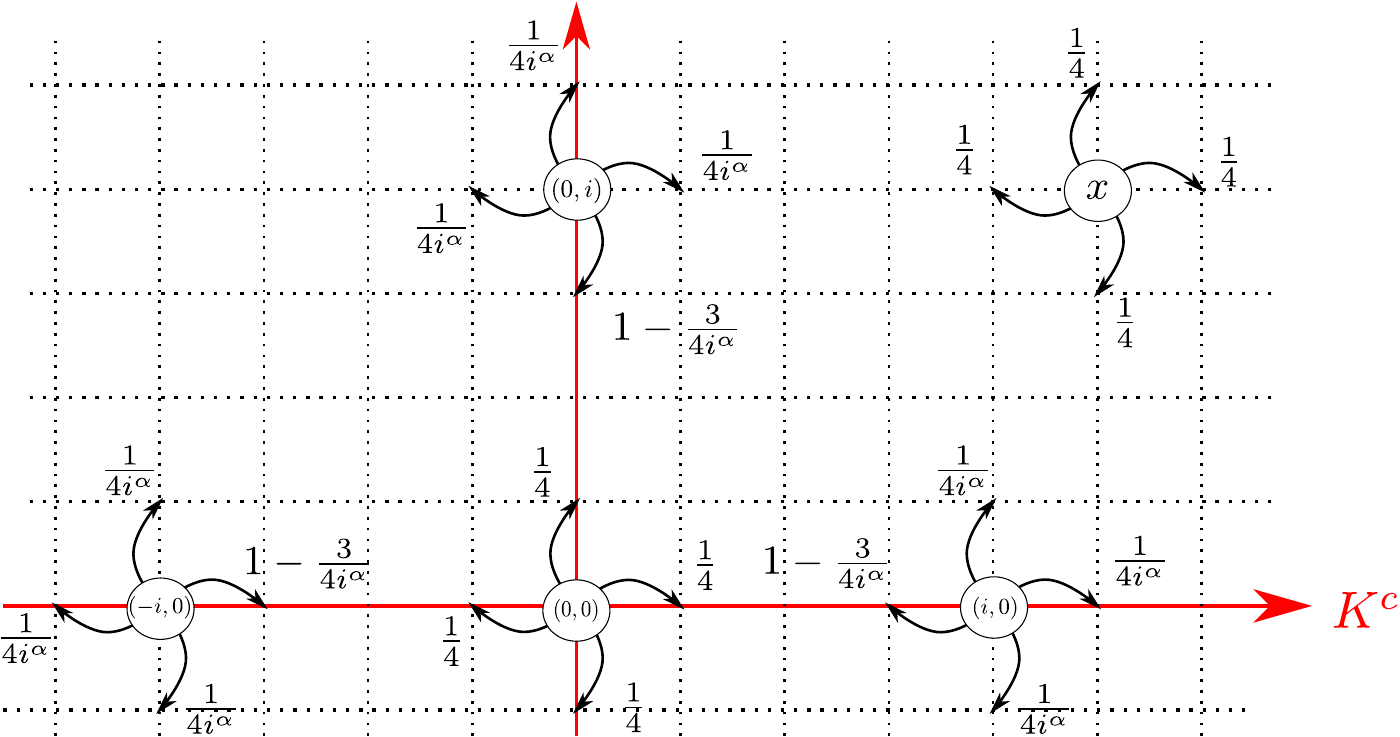}
\caption{probabilities of transition }
\label{whynot}
\end{figure}
The fluctuations  of this random walk can be seen as the movement of a particle which is diffusive on the cones but which is submitted to a field on the axis, this field, which source comes from infinity, is decreasing when the particle gets closer to the origin. Depending on the strength of the local field, represented by $\alpha$, the behavior of this random walk is more or less perturbed comparing to the simple random walk. In this first work we deal with the case for which the applied force toward $0$ is strong, that is $\alpha>3$.  To give an idea on the changes that provokes such a perturbation, we first present two examples of our main result that are related to the local time of the walk: for all subset $A$ of $\Z^2$, $\mathscr{L}(A,n)=\sum_{i=1}^{n}\un_{X_i \in A} $ is the local time of $\mathbf{X}$ in $A$ until time $n$. Our first result states as follows

\begin{theo}
 Assume $\alpha> 3$, there exists two positive constants $\cn$ and $\cn'$ such that:
\begin{align*}
 \frac{\log n}{n}\mathscr{L}((0,0),n)  \overset{\P}{ \rightarrow } \cn',\quad \frac{\log n}{n}\mathscr{L}(K^c,n)  \overset{\P}{ \rightarrow } \cn.
\end{align*}
\end{theo}
\noindent Note that $\cn$ and $\cn'$ are explicitly given a little further.\\
\noindent  Clearly the behavior of the local time of $\mathbf{X}$ is very different than  the one of the symmetric random walk  $\mathbf{S}$. 
For example, $\mathscr{L}_{\mathbf{S}}((0,0),n)/ \log n $ converges in law to an exponential variable with parameter $\pi$ (see for instance \cite{Revesz}, \cite{Erdor}).\\
{It turns out that these two results are just simple consequences of a more general result presented below. It is essentially an ergodic-like theorem for the part of the trajectory on the axis. Recall that we assume $X_0=(1,1)$ and let us introduce the following stopping times: for any $i \in \N^*$
\begin{align}
 \eta_i & := \inf \{k>\rho_{i-1},  X_{k} \in K^c \}, \nonumber \\ 
\rho_{i}& := \inf \{k>\eta_{i},  X_{k} \in K \},  \rho_0=0, \nonumber
\end{align}
that is respectively the $i$-th  exit and entrance times in $K$.
Also for any $n$ let 
\begin{align}
N_n& := \max\{i \leq n,\ \rho_i \leq n \} \label{last},
\end{align}
which is the index of the last excursion to $K$ before the instant $n$.\\
Moreover, for any $x=(x_1,x_2) \in \Z^2$, let us denote by $\ox:=\max(|x_1|,|x_2|)$, the maximum norm of $x$.\\
Our main result is the following:
\begin{theo} \label{mainth} Let 
\[\mathscr{B}_i:=(X_k, \eta_i \leq k < \rho_i)\ \textrm{ and }\   \mathscr{B}^*_n:=(X_k,  \eta_{N_n} \leq k \leq   n  ), \] the portions of  the trajectory of $\bf{X}$ restricted to the axis. 
  Let  $f$ a positive non-decreasing functional (that is for any $i \leq k$, $f(x_1, \cdots,x_i) \leq f(x_1, \cdots,x_k)$) such  that  there exist $0<\delta<2$ and two positive constants $C_1$ and $C_2$ such that for any  $x \in K^c$,
   \begin{align}
 & \E_x\left[f(\mathscr B_0)\right] \leq  C_1 \ox, \label{cond2}\\
 &{\mathbb V}\mathrm{ar}_x\left[f(\mathscr B_0) \right]
 \leq C_2 \ox^{2-\delta}, \label{cond1}
  \end{align} 
where $\mathscr B_0:=(X_k,0\le k<\rho)$ and $\rho:= \inf \{k>0, X_k \in K\}$. Assume $\alpha> 3$ then, in probability
\begin{align}
\lim_{n\rightarrow+\infty}\frac{\log n}{n} \left( \sum_{i=1}^{N_n}f( \mathscr{B}_i)+f( \mathscr{B}_n^* )\right) =\lim_{n\rightarrow+\infty}\frac{\log n}{n} \sum_{i=1}^{N_n}f( \mathscr{B}_i)= \cn^f,
\end{align}
where $ \cn^f$ is a positive constant that is described below.
\end{theo} 
\noindent Let us give some underlying ideas concerning our main result: the trajectory of $\bf X$ can be split in excursions composed of parts on the axis and parts on the cones. For the parts on the axis, we can prove that at each excursion (before the instant $n$) the walk escapes from $K^c$ in a compact neighborhood of $(0,0)$ with a probability close to one. This escape coordinate is actually driven by a first  invariant probability measure.  The second part on $K$ yields that the number of excursions (composed of the part on the cone and the part on the axis)  before the instant $ n $ is of order $n/ \log n$. This fact comes essentially from the tail of the first time the walk exits a cone when starting from a coordinate of the neighborhood of $(0,0)$. It turns out that exit coordinates of the cone are also driven by a second invariant probability measure.  This finally makes appears a renewal structure for the trajectory of the walk.

\noindent Let us now discuss about the constant which appears in Theorem \ref{mainth}
\begin{align}
\cn^f:=\frac{\pi}{
8} \frac{\sum_{y} \E_y[f(\mathscr B_0)] \pi^*(y)} {\sum_{x}\overline x \pi^{\dag}(x) } =\frac{\pi}{8}\frac{\E_{\pi^*}[f(\mathscr B_0)]}{\E_{\pi^\dag}[X_0]},
\end{align}
where $\pi^*$ and $\pi^{\dag}$ are respectively the unique invariant probability measures of the Markov chains $(X_{\eta_i},i)$ and $(X_{\rho_i},i)$. Note that by definition $\pi^*(x)\neq 0 \iff x \in  K^c $ whereas $\pi^{\dag}(x) \neq 0 \iff x \in \partial K:=\lbrace x\in K \slash \exists y\in K^c, \overline{x-y}=1\rbrace $. We can actually give some details about these two measures, in particular their tail: $\lim_{\ox \rightarrow +\infty } \ox^{3} \pi^*(x)=\co>0$ and $\lim_{\ox \rightarrow +\infty} \ox^{\alpha+2} \pi^{\dag}(x)=\cde $, both  constants $\co$ and $\cde$ are partially explicit (see respectively Proposition \ref{prop1} and Lemma \ref{lemrec2}). 
\noindent With this theorem and the fact that families of function $f$ such that for any $k<m$, $f(x_k,\cdots,x_m)=\sum_{i=k}^m \un_{x_i\in K^c}$ (to obtain the local time on the axis) or $f(x_k,\cdots,x_m )= \sum_{i=k}^m \un_{x_i=(0,0)}$ (to obtain the local time at the origin) satisfy \eqref{cond2} and \eqref{cond1}, we obtain the first theorem. This also leads to the constants $\cn$ and $\cn'$ from above expression of $\cn^f$. Note that the increasing hypothesis on $f$ is only used to prove that  $\frac{\log n}{n}  f( \mathscr{B}_n^*)$ converges in probability to zero and is unnecessary to prove the convergence of  $\frac{\log n}{n} \sum_{i=1}^{N_n}f( \mathscr{B}_i)$ to $\cn^f$. \\

We can think of several extensions and possible generalizations of this result.
The first question is what happen when return-force is lower, that is $\alpha \leq 3$. Well in this case the renewal structure does not exist any more, and things worsen when $\alpha<2$ as there is no longer concentration of $(X_{\rho_i},i)$ in the neighborhood of $(0,0)$ so a totally new approach is needed. \\
Possible generalization concerns first the shape of the return force, we could easily replace function $i \mapsto {i^{-\alpha}}$
 by a sequence $(a_i,i)$ such that $\sum_{i} i^2 a_i <+ \infty$, note that this modification should not change the results. However we chose to keep $i^{-\alpha }$ in order to obtain proof easily readable. For the random walk on $K$, we could also consider more general random walks, this would need improvement of local limit theorems like the ones proved for simple random in Section \ref{lecone} (Lemmata \ref{MarcheSimple} and \ref{MarcheSimpleTemps}). More specifically what is needed, for example, is uniform convergence for large $x$ of $(\P_y(X_{\eta}=x),y)$ (where $\eta:=\inf\{k>0, X_k \in K\}$, see Lemma \ref{MarcheSimple}). Note that $\P_y(X_{\eta}=x)$ is studied for several random walks in \cite{Raschel}, including simple random walk, but some extra work is needed to obtain uniform convergences.\\

The paper is organized as follows, in the following section we prove an ergodic result for the first $m$-excursions, in Section \ref{sec3a} we prove that the number of excursions before the instant $n$ is of order $n/ \log n$ and finish with the proof of Theorem \ref{mainth}. Finally in Section \ref{sec4} (resp. Section \ref{lecone}) we resume stochastic estimations for the walk when it remains on the axis (resp. on the cones). For the seek of completeness, we had an appendix in Section  \ref{sec6}.\\
Note that although our main result needs the assumption $\alpha>3$, we allege this hypothesis whenever it is possible as this will be useful for future works (see more specifically Section \ref{sec4} which only deals with the trajectory of the walk on the axis).

\section{ Ergodic result on the axis during the first $m$  two-types excursions.}

We call $i^{th}$ two-types excursion the trajectory of $X$ on the time interval $\rho_{i-1} < k \leq  \rho_i$, the first part  concerning the cones  and the second one the axis.  In this section we prove 
that for any positive non-decreasing functional $f$ satisfying \eqref{cond2} and \eqref{cond1}, its empirical mean along the trajectory of $X$ on  ${K}^c$ during the first $m$ excursions converges. Recall that for $i\ge 1$, $\mathscr{B}_i=(X_k, \eta_i \leq k < \rho_i)$ and $\mathscr{B}_0=(X_k, 0 \leq k < \rho)$.

\begin{prop} \label{prop1} Assume $\alpha> 3$, there exists a probability measure $\pi^*$ such that
\begin{align*}
\frac{1}{m} \sum_{i=1}^m f(\mathscr{B}_i)   \overset{\P}{ \rightarrow }  \sum_{{x \in K^c}} \pi^*(x)\E_x\left[f(\mathscr B_0)\right]=\E_{\pi^*}[f(\mathscr B_0)], 
\end{align*}
moreover  $\lim_{\ox \rightarrow +\infty } \ox^{3} \pi^*(x)={\mathbf{c_{_0}}}>0$, with $\co:=\frac{16}{\pi}\sum_{y } \pi^*(y)  \E_y[X_{\rho} ].$
\end{prop}

\noindent To obtain this proposition, we prove two Lemmata, the first one tells that $\sum_{i=1}^m f(\mathscr{B}_i) $ can be approximated by $\sum_{i=1}^m  \E_{X_{\eta_i}}[f(\mathscr B_0)]$,  the second studies the convergence of   this last sum. First Lemma states as follows 
\begin{lem} \label{lem3.2} Assume $\alpha \ge  3$, there exists $0<C^{\prime}<+ \infty$, such that
\begin{align}\label{six}
\E\left[ \frac{1}{m^2} \left( \sum_{i=1}^m  f(\mathscr{B}_i) - \E_{X_{\eta_i}}[f(\mathscr B_0)]  \right)^2\right] \leq \frac{C^\prime}{m}.
\end{align}
\end{lem}

\begin{proof}
Strong Markov property leads to 
\begin{align*} 
 \E\left[ \left( \sum_{i=1}^m  f(\mathscr{B}_i) - \E_{X_{\eta_i}}[f(\mathscr B_0)]  \right)^2 \right] 
=\sum_{i=1}^m \E\left[ \V_{X_{\eta_i}}[f(\mathscr B_0 )]\right].
\end{align*}
By hypothesis \eqref{cond1}, above quantity is smaller than $C_2 \sum_{i=1}^m \E[\overline{X }_{\eta_i}^{2-\delta}]$ with $0<\delta<2$ and we conclude using Lemma \ref{newlem51} telling that there exists $C>0$ such that $\E\left[\overline{ X}_{\eta_i}^{2-\delta}\right]<C$ for all $i\in \mathbb N^*$. 
\end{proof}

\noindent \\ Second Lemma, which is essentially an ergodic result, writes
\begin{lem} \label{lem3.5} Assume $\alpha>3$, the Markov chain  $(X_{\eta_i},i)$ is positive recurrent and its invariant probability measure $\pi^*$ satisfies 
\begin{align}\label{Birkh1}
\frac{1}{m} \sum_{i=1}^m  \E_{X_{\eta_i}}[f(\mathscr B_0)]\overset{\P.a.s.}{ \rightarrow }  \sum_{x\in K^c} \pi^*(x)\E_x[f(\mathscr B_0)]=\E_{\pi^*}[f(\mathscr B_0)].
\end{align}
\end{lem}

\begin{proof}  $(X_{\eta_i},i)$ being obviously irreducible,  we just have to prove that $(0,1)$ is positive recurrent. Introduce, for any $x\in K^c$,  $\tau_x=\inf\lbrace k\ge 0, X_{\eta_k}=x \rbrace$, then for all $k>0$: 
\begin{align*}
\P_{(0,1)}\left(\tau_{(0,1)}>k\right)&=\P\left(\forall i\le k,  X_{\eta_i}\ne (0,1)\right)\\ 
&=\sum_{y\in, K^c\backslash \lbrace(0,1)\rbrace}\P_{(0,1)}\left(\forall i\le k-2,  X_{\eta_i}\ne (0,1),X_{\eta_{k-1}}=y\right)\left(1-\P_y\left(X_{\eta_1}= (0,1)\right)\right).
\end{align*} 
According to \eqref{infcone}, there exists $0<C<1$ such that for all $y\in K^c,\, \P_y\left(X_{\rho}=(1,1)\right)>C$, implying:
\begin{align*}
\P_y\left(X_{\eta_1}= (0,1)\right)
\ge \P_{y}\left(X_\rho=(1,1)\right)\P_{(1,1)}\left(X_{\eta}=(0,1)\right)
\ge  C\P_{(1,1)}\left(X_{1}=(0,1)\right)=\frac{C}{4}.
\end{align*}
Consequently, with an obvious induction reasoning: 
\begin{align*}
\P_{(0,1)}\left(\tau_{(0,1)}>k\right)&\le \sum_{y\in K^c\backslash \lbrace(0,1)\rbrace}\P_{(0,1)}\left(\forall i\le k-2,  X_{\eta_i}\ne (0,1),X_{\eta_{k-1}}=y\right)\left(1-\frac{C}{4}\right)\\
&=\P_{(0,1)}\left(\tau_{(0,1)}>k-1 \right)\left(1-\frac{C}{4}\right)\le \left(1-\frac{C}{4}\right)^k.
\end{align*} 
Thus  $\E_{(0,1)}\left[\tau_{(0,1)} \right]=\sum_{k\ge0}\P_{(0,1)}\left(\tau_{(0,1)}>k\right)
<\infty$ and  $(X_{\eta_i},i)$ is positive recurrent. \\
\eqref{Birkh1} is an application of Birkhoff's ergodic Theorem so we only have to check that 
 $\E_{\pi^*}[f(\mathscr B_0)]  $ exists:  first note that by condition  \eqref{cond2} for any $x$, $\E_x[f(\mathscr B_0)] \leq  C_+ \ox$, so we only have to check that $ \sum_{x} \ox \pi^*(x)<+\infty $.  For that, we have to study the asymptotic in $\ox$ of $\pi^*(x)$. Let $y$ in $K^c$ and $\delta>0$ small enough such that $(1-\delta) \alpha>3$, for any $x$  
  \begin{align}
\P_y(X_{\eta_1}=x)=\sum_{\overline{z} \leq \ox^{1- \delta} } \P_y(X_{\rho}=z)\P_z(X_{\eta}=x)+ \sum_{\overline{z} > \ox^{1- \delta}} \P_y(X_{\rho}=z)\P_z(X_{\eta}=x). \label{eq10}
\end{align}
 By \eqref{easypeasy}, there exists a positive constant $c_+$ such that for all $z\in K$, $  \P_y(X_{\rho}=z) \leq (1+c_+) \oz^{-\alpha} $,  so using Lemma \ref{reverse}, the second sum above is bounded by 
  \begin{align}
\!\!\!\! (1+c_+)\sum_{\overline{z} > \ox^{1- \delta}} \oz^{- \alpha}\P_z(X_{\eta}=x) &\leq  (1+c_+) \ox^{- (1- \delta)\alpha}\!\! \sum_{\overline{z} > \ox^{1- \delta}} \P_z(X_{\eta}=x)  
\leq 2 (1+c_+) \ox^{- (1- \delta)\alpha} =o(\ox^{-3}).\label{eq11}
\end{align}
 Local limit result (Lemme \ref{MarcheSimple}) implies that $(\ox^3/\oz) \P_z(X_{\eta}=x) \sim 16/\pi  $ for any large $x$ uniformly in $z$ with $\oz \leq \ox^{1- \delta}$,   so for the  first sum in \eqref{eq10} we get for large $\ox $
 \begin{align}
 \sum_{\overline{z} \leq \ox^{1- \delta} } \P_y(X_{\rho}=z)\P_z(X_{\eta}=x) \sim  \frac{16}{\pi\ox^3}  \sum_{\overline{z} \leq \ox^{1- \delta} }   \oz    \P_y(X_{\rho}=z) =  \frac{16}{\pi\ox^3} \E_y[\oX_{\rho} \un_{X_{\rho} \leq  \ox^{1- \delta} }].\label{eq12}
\end{align}
Then \eqref{eq10}, \eqref{eq11} and \eqref{eq12} implies that for any $y$, 
\begin{align*}
\lim_{\ox\rightarrow+\infty}\ox^3  \P_y(X_{\eta_1}=x) =\frac{16}{\pi}  \E_y[\oX_{\rho}]\le M
\end{align*}
 where $M$ does not depend on $y$ by Corollary  \ref{mean}. As $\pi^*(x)=\sum_{y\in K^c}\pi^*(y)\P_y(X_{\eta_1}=x)$:
\[ \lim_{\ox \rightarrow +\infty}  \ox^{3 } \pi^*(x)=\sum_{y } \pi^*(y)  \lim_{\ox \rightarrow +\infty } \ox^{3 } \P_y(X_{\eta_{1}} = x )=\frac{16}{\pi} \sum_{y } \pi^*(y)  \E_y[\oX_{\rho} ]\le \frac{16 M}{\pi}.\]  
So $ \sum_{x}\ox \pi^*(x)<+\infty $ is satisfied. 
\end{proof}

\noindent  Lemmata \ref{lem3.2} and \ref{lem3.5} ensure that $\frac{1}{m} \left( \sum_{i=1}^m  f(\mathscr{B}_i) - \E_{X_{\eta_i}}[f(\mathscr B_0)]\right)$ tends to 0 and $\frac{1}{m}  \E_{X_{\eta_i}}[f(\mathscr B_0)]$ to $\E_{\pi^*}[f(\mathscr B_0)]$ in probability which gives Proposition \ref{prop1}.
It yields following Corollary giving the behaviors (after $m$ excursions) of the local time on the axis and at $(0,0)$. 

\begin{cor} \label{cor1} Assume $\alpha> 3$,
\begin{align*}
\frac{1}{m} \lo(K^c, \rho_m) \overset{\P}{ \rightarrow }  \sum_{x} \pi^*(x)\E_x[\rho] \textrm{ and } \frac{1}{m} \lo((0,0), \rho_m)
 \overset{\P}{ \rightarrow }  \sum_{x} \pi^*(x)\E_x[\lo(0,\rho)]. 
\end{align*}
\end{cor}
\begin{proof}
As $\lo(K^c, \rho_m)= \sum_{i=1}^m (\rho_i-\eta_i) $, if we take, for any $k<m$, $f(x_k,\cdots,x_m)=m-k$, then $\lo(K^c, \rho_m)= \sum_{i=1}^m f(\mathscr{B}_i) $. So to prove the result for $\lo(K^c, \rho_m)$ we only have to check that conditions  \eqref{cond2} and \eqref{cond1} are full-filled for this $f$. By Lemma \ref{lem6.2}, for any $\gamma \in \{1,2\}$ there exists $0<\varepsilon<1$ such that for any $x\in K^c$:
\[\E_{x} [f^\gamma(\mathscr B_0)] -\E_{x} [f(\mathscr B_0 )]^\gamma  =\E_{x} [\rho^\gamma ] -\E_{x} [\rho ]^\gamma \leq  \ox^{\gamma-\varepsilon} , \]
 so both \eqref{cond2} and \eqref{cond1} are satisfied. 
$\lo((0,0), \rho_m)=\sum_{i=1}^m (\lo((0,0),\rho_i)-\lo((0,0),\eta_i)) $ is treated similarly.
\end{proof}

\section{The number of two-types excursions before the instant $n$ \label{sec3a}}
 
In this section we prove following Proposition, recall that $N_n$ is the number of the last entrance in $K$ before $n$ (see  \eqref{last}).
 \begin{prop} \label{prop31}
 Assume $\alpha>3$, there exists a probability measure $\pi^{\dag}$ such that
 \begin{align}
\frac{\log m}{m}N_m  \overset{\P}{ \rightarrow } \frac{\pi}
{8}\frac{1}{  \sum_{x}\overline x \pi^{\dag}(x)}=:\cun.
\end{align}
 \end{prop}
 
 \noindent The main idea comes from decomposition of $\rho_k=\sum_{i=1}^k \rho_i-\eta_{i}+\sum_{i=1}^k \eta_i-\rho_{i-1}$, then by Corollary \ref{cor1}, $\sum_{i=1}^k \rho_i-\eta_{i}$ is of order $k$ whereas we show below that in probablity $\sum_{i=1}^k \eta_i-\rho_{i-1}$ is of order $k \log k$. This last fact comes from the tail of $\eta$ (the single return instant to the axis) together with the fact that for any $i \leq k$, with an overwhelming probability $X_{\rho_i}$ is in the neighborhood of $(0,0)$ which size is independent of $i$ and of the coordinates of entry of the walk on the axis. We start with following Lemma 
 
 \begin{lem} \label{lemrec2}
 Assume  $\alpha >3$, $(X{\rho_i},i)$ is positive recurrent. Moreover its invariant probability measure $\pi^{\dag}$ satisfies $\lim_{\ox \rightarrow +\infty} \ox^{\alpha+2} \pi^{\dag}(x)=\cde $ with $\cde= \frac{2}{ \pi} \sum_{u}\ou \pi^{\dag}(u)$.
\end{lem}

\begin{proof}
The proof of the fact that $(X_{\rho_i},i)$ is positive recurrent is very similar than for $(X_{\eta_i},i)$. Using the irreducibility of this chain, it suffices to prove that $(1,1)$ is positive recurrent. Denote by $\tilde \tau_A=\inf\lbrace k> 0, X_{\rho_k}\in A\rbrace$, then, for all $k>0$: 
\begin{align*}
\P_{(1,1)}\left(\tilde \tau_{(1,1)}>k\right)
&=\sum_{y\in K\backslash \lbrace(1,1)\rbrace}\P_{(1,1)}\left(\forall i\le k-2,  X_{\rho_i}\ne (1,1),X_{\rho_{k-1}}=y\right)\left(1-\P_y\left(X_{\rho_1}= (1,1)\right)\right).
\end{align*} 
Using \eqref{infcone}, there exists $0<C<1$ such that for all $z\in K^c,\, \P_z\left(X_{\rho}=(1,1)\right)>C$, implying that
 $\P_y\left(X_{\rho_1}= (1,1)\right)\ge C \sum_{z\in K^c}\P_y\left(X_{\eta}=z\right)=C$. So, with an induction reasoning:
 \[\E_{(1,1)}[\tilde \tau_{(1,1)}]=\sum_{k\ge 0}\P_{(1,1)}\left(\tilde \tau_{(1,1)}>k\right)
 \le\sum_{k\ge 0} \left(1-C\right)^k<\infty.\] 
Let us study the tail of invariant probability measure $\pi^{\dag}$. First, note that as $(X_{\rho_i},i)$ is obviously aperiodic, using Lemma \ref{lem410}, for all $x\in K$: 
\begin{align}\label{infdag}
\pi^\dag(x)=\lim_{n\rightarrow+\infty}\P(X_{\rho_n}=x)\leq \frac{c'}{\ox^{\alpha}}
\end{align}
implying that $\pi^\dag$ has a first moment and $\cun\ne 0$.
Let $\delta>0$ small enough such that $(\alpha-1)(1- \delta)>2$,  assume $x=(1,x_2)$ with $x_2>0$ (other cases can be treated similarly), let $L_x:=\{z = (0,z_2)\ with\ z_2 \geq x_2 \}$ and $L_x^c$ the relative complement of $L_x$  in $K^c$. As $\pi^\dag(x)=\sum_{y\in K}\pi^\dag(y)\P_y(X_{\rho_1}=x)$: 
\begin{align} 
&\pi^\dag(x)= \sum_{y} \pi^\dag(y) \sum_{z \in L_x} \P_y(X_\eta=z)\P_z(X_{\rho}=x)+ \sum_{y} \pi^\dag(y) \sum_{z \in L_x^c} \P_y(X_\eta=z)\P_z(X_{\rho}=x)=:S_1+S_2. \label{la13}  
\end{align}
When $z \in L_x^c$, by the second point of Corollary \ref{cor4.6}, there exists $C^\prime>0$ such that $\P_z(X_{\rho}=x)\leq C^\prime/ \ox^{ 2 \alpha}$, so as $\alpha>3$
\begin{align*} 
 S_2\le \frac{C^\prime}{\ox^{ 2 \alpha}} \sum_{y\in K} \pi^\dag(y) \sum_{z \in L_x^c} \P_y(X_\eta=z)\leq \frac{C^\prime}{\ox^{ 2 \alpha}}=o(\ox^{-2-\alpha}).  
\end{align*}
We now deal with the first sum in \eqref{la13} which we decompose as follows 
\begin{align*} 
 S_1 
= \sum_{y, \oy \leq \ox^{1- \delta}} \pi^\dag(y) \sum_{z \in L_x} \P_y(X_\eta=z)\P_z(X_{\rho}=x)
 +\sum_{y, \oy > \ox^{1- \delta}} \pi^\dag(y) \sum_{z \in L_x} \P_y(X_\eta=z)\P_z(X_{\rho}=x)=: \Sigma_1+ \Sigma_2.
\end{align*}
For $\Sigma_2$, by \eqref{easypeasy}
and \eqref{infdag}, as $\alpha+(\alpha-1)(1- \delta)>\alpha+2$: 
\[\Sigma_2\le \frac{C_+}{\ox^\alpha}\sum_{y, \oy > \ox^{1- \delta}} \frac{1}{\oy^\alpha}\le  \frac{C_+}{\ox^{\alpha+(\alpha-1)(1- \delta)}}=o(\ox^{-2-\alpha}),\] 
where $C_+$ is a positive constant that may grow from line to line. In view of what we want to prove, $S_2$ and $\Sigma_2$ are negligible.\\ 
For $\Sigma_1$,  we use, as for $\Sigma_2$, \eqref{easypeasy}, the first point of Corollary \ref{cor4.6} and also Lemma \ref{MarcheSimple} telling that uniformly in $y$, with $\oy \leq \ox^{1- \delta}$ and $z \in L_x $, $ \P_y(X_\eta=z) \sim 16 \oy / \oz^3 \pi$, from this we deduce that 
\begin{align*} 
 \Sigma_1\sim \sum_{y, \oy \leq \ox^{1- \delta}} \pi^\dag(y)   \sum_{z \in L_x} \frac{4 \oy }{\pi \oz ^3 \ox ^{\alpha}}\sim \frac{2}{\pi\ox ^{\alpha+2}} \sum_{y, \oy \leq \ox^{1- \delta}} \oy \pi^\dag(y) . 
\end{align*}
Finally as $\pi^\dag$ has a first moment, $ \lim_{\ox \rightarrow +\infty} \ox ^{\alpha+2}  \Sigma_1 = \frac{2}{\pi}\sum_{u}\ou \pi^{\dag}(u)$, this finishes the proof.
\end{proof}

\noindent Second Lemma below is a law of large number for the time spent by the walk on the cone during the first $m$ excursions.
 \begin{lem} \label{lem3.3bis} Assume $\alpha>3$, then in probability 
\begin{align*}
\lim_{m \rightarrow + \infty} \frac{1}{m \log m}\sum_{i=1}^m(\eta_i-\rho_{i-1})= \frac{8}{\pi} \sum_{x\in K} \overline x \pi^{\dag}(x).
\end{align*}
\end{lem}

\begin{proof}
For any $0<\varepsilon<1$, let us decompose $\sum_{i=1}^m(\eta_i-\rho_{i-1})$ as follows
\begin{align*}
\sum_{i=1}^m (\eta_i-\rho_{i-1})=\sum_{i=1}^m (\eta_i-\rho_{i-1})\un_{\eta_i-\rho_{i-1} \leq \varepsilon m} + \sum_{i=1}^m (\eta_i-\rho_{i-1})\un_{\eta_i-\rho_{i-1} > \varepsilon m}=:\Sigma_1+ \Sigma_2.  
\end{align*}
\noindent $\bullet$ { Let us prove that $\Sigma_2=o(m\log m)$: let $\mathscr{A} :=\{\sum_{i=1}^m \un_{\eta_i-\rho_{i-1} >  \varepsilon m} \leq  (\log m)^{1/4} \}$,  Markov inequality gives 
\begin{align*}
\P\left(\mathscr{A}^c \right) \leq \frac{1}{(\log m)^{1/4}} \sum_{i=1}^m \P(\eta_i-\rho_{i-1}> \varepsilon m ).
\end{align*}
For $0<\delta<\nicefrac{1}{2}$, one can write: 
\begin{align*}
\P(\eta_i-\rho_{i-1}> \varepsilon m ) &=\sum_{x\in K} \P(X_{\rho_{i-1}}=x)\P_x(\eta> \varepsilon m) \\ 
&=\sum_{x, \ox \leq m^{1/2- \delta}} \P(X_{\rho_{i-1}}=x)\P_x(\eta> \varepsilon m) +\sum_{x, \ox > m^{1/2- \delta}} \P(X_{\rho_{i-1}}=x)\P_x(\eta> \varepsilon m).
\end{align*}
According to Lemma \ref{MarcheSimpleTemps}, uniformly in $x$ such that $\ox\le m^{\nicefrac{1}{2}-\delta}$ for $m$ large enough  $\P_x(\eta> \varepsilon m) \sim 8 \ox / \pi \varepsilon m$. Then for $m$ large enough and as $\alpha>3$ using Lemma \ref{lem410}:
\[\sum_{x, \ox \leq m^{1/2- \delta}} \P(X_{\rho_{i-1}}=x)\P_x(\eta> \varepsilon m) \leq  \frac{8}{\varepsilon m} \E[X_{\rho_{i-1}}]     \leq  \frac{C_+}{\varepsilon m}.\]
As $\alpha>3$, we can chose $\delta$ such that $(\alpha-1)(1/2-\delta)> 1$ and using again Lemma \ref{lem410}:
\[\sum_{x,\ox > m^{{1}/{2}- \delta}} \P(X_{\rho_{i-1}}=x)\P_x(\eta> \varepsilon m) \leq  \sum_{x, \ox > m^{1/2- \delta}} \frac{c'}{\ox^{\alpha}} \leq \frac{c'}{m^{(\alpha-1)(1/2-\delta)}} =o\Big( \frac{1}{m}\Big).\]
Then $\P(\eta_i-\rho_{i-1}> \varepsilon m ) \leq C_+ / \varepsilon m$,  so finally  $\P(\mathscr{A}^c) \leq C_+ / (\log m)^{1/4} \varepsilon $. A similar computation also prove that $\P(\mathscr{C}^c) \leq C_+ /   (\log m)^{1/2}$  with  $\mathscr{C}:=\{ \sum_{i=1}^m \un_{\eta_i-\rho_{i-1} >  m (\log m)^{1/2} }    =0 \} $. Now, notice that on $\mathscr{A} \cap \mathscr{C}$, 
\begin{align*}
\Sigma_2=\sum_{i=1}^m (\eta_i-\rho_{i-1})\un_{\eta_i-\rho_{i-1} > \varepsilon m}\le m (\log m)^{1/2}  \sum_{i=1}^m \un_{\eta_i-\rho_{i-1} > \varepsilon m} \leq   m (\log m)^{3/4},
\end{align*}
which implies $\P\left(\Sigma_2>  m (\log m)^{3/4} \right)  \leq \nicefrac{C_+} {((\log m)^{1/4} \varepsilon)}$,
and thus, in probability $\Sigma_2=o(m \log m)$.} \\
\noindent $\bullet$ For $\Sigma_1$, assume for the moment that  
\begin{align}
\lim_{m \rightarrow + \infty}\P\Big(\Big|\Sigma_1-\sum_{i=1}^m \E_{X_{\rho_{i-1}}}[\eta \un_{\eta \leq \varepsilon m}]\Big|>  m (\log m)^{1/2} \Big)=0 \label{neglpart}.
\end{align}
Let us compute:
\begin{align}
\E_{X_{\rho_{i-1}}}\left[\eta \un_{\eta \leq \varepsilon m }\right]& =\E_{X_{\rho_{i-1}}}\left[\eta \un_{\eta \leq (\log m)^{1/2}}\right]+\E_{X_{\rho_{i-1}}}\left[\eta \un_{(\log m)^{1/2}<\eta \leq \varepsilon m }\right], \label{Eq15}
\end{align}
the first sum is smaller than $(\log m)^{1/2}$. 
For the second one, we decompose : 
\begin{align}
\E_{X_{\rho_{i-1}}}\left[\eta \un_{(\log m)^{1/2}<\eta \leq \varepsilon m }\right] &= \sum_{ (\log m)^{1/2} < k \leq \varepsilon  m } \P_{X_{\rho_{i-1}}}(k<\eta\le \varepsilon m) (\un_{\oX_{\rho_{i-1}} \leq k^{1/2-\delta}} +  \un_{\oX_{\rho_{i-1}} > k^{1/2-\delta}})\nonumber\\
&=:\Sigma_3+ \Sigma_4. \label{Eq16}
\end{align}
In order to simplify the writing in the sequel, we introduce the following inequality: for $(a,b)$ such that $a<\alpha-1$ and $(\nicefrac{1}{2}-\delta)(\alpha-a-1)-b>1$, using Lemma \ref{lem410}:
\begin{align*}
\E[F(a,b)]&:=\E\Big[\sum_{ (\log m)^{1/2} < k \leq \varepsilon  m }k^b\oX_{\rho_{i-1}}^a\un_{\oX_{\rho_{i-1}} > k^{1/2-\delta}}\Big]
\le C_+ \sum_{ (\log m)^{1/2} < k \leq  \varepsilon m} k^{b} \sum_{r > k^{1/2-\delta} } r^{a- \alpha}\\
&\le  C_+ \sum_{ (\log m)^{1/2} < k \leq  \varepsilon m} k^{b+(a- \alpha+1)(\nicefrac{1}{2}-\delta)}\le C_+ (\log m)^{\frac{1}{2}(b+1+(a- \alpha+1)(\nicefrac{1}{2}-\delta))}.
\end{align*}
For $\Sigma_3$, we use Lemma \ref{MarcheSimpleTemps}: for any  $\oz \leq k^{1/2-\delta}$,  $\P_{z}( \eta =k ) \sim   8 \overline z/ \pi k^2$, so for large $m$,  
\begin{align}
\Sigma_3&= 
  (1+o(1)) \frac{8}{\pi}  \overline{X}_{\rho_{i-1}}   \sum_{ (\log m)^{1/2} < k \leq \varepsilon  m }\frac{1}{k} \un_{\oX_{\rho_{i-1}} \leq k^{1/2-\delta}} \nonumber 
\\
&= \frac{8}{\pi} \overline{X}_{\rho_{i-1}}  (1+o(1))  \left(\sum_{ (\log m)^{1/2} < k \leq \varepsilon  m } \frac{1}{k} - \sum_{ (\log m)^{1/2} < k \leq \varepsilon  m } \frac{1}{k}\un_{\oX_{\rho_{i-1}} > k^{1/2-\delta}}\right)  \nonumber \\
&= \frac{8}{\pi}  (\log m) \overline X_{\rho_{i-1}}  (1+o(1))  - \frac{8}{\pi}  {\overline X_{\rho_{i-1}}} (1+o(1))  \sum_{ (\log m)^{1/2} < k \leq  \varepsilon m}\frac{1}{k}\un_{\oX_{\rho_{i-1}} > k^{1/2-\delta}}. \label{Eq17}
\end{align}
Finally uniformly in $i\leq m$, as $\Sigma_4\le F(0,0)$:
\begin{align*}
\log m  \overline X_{\rho_{i-1}} (1+o(1))-   2F(1,-1)\le     \frac{\pi}{8}\E_{X_{\rho_{i-1}}}\left[\eta \un_{(\log m)^{1/2}<\eta \leq \varepsilon m }\right]\le \log m \overline X_{\rho_{i-1}} (1+o(1))+\frac{\pi}{8}F(0,0).
\end{align*}
Thanks to our choice of $\delta$, $\E[F(0,0)]$ and $\E[F(1,-1)]$ tend to zero, implying (using a Markov inequality) that
\begin{align}
\lim_{m \rightarrow +\infty} \P\Big(\sum_{i=1}^m \sum_{ (\log m)^{1/2} < k \leq \varepsilon  m }({\overline X_{\rho_{i-1}}} k^{-1}\un_{\oX_{\rho_{i-1}}> k^{1/2-\delta}}+\un_{\oX_{\rho_{i-1}}> k^{1/2-\delta}})>m \Big)=0. \label{Eq18}
\end{align}
One can notice that $i=1$ is a special case  as $X_{\rho_0}=(1,1)$ a.s. and that we should have made its own reasoning. However, as the estimates remain true with even simpler calculations, we have decided to not put it.\\
Collecting \eqref{Eq15}, \eqref{Eq16}
, \eqref{Eq17} and \eqref{Eq18} we obtain that, in probability, when $m$ tends to infinity 
 \begin{align*}
 \frac{1}{m \log m} \sum_{i=1}^m 
 \E_{X_{\rho_{i-1}}}[\eta \un_{\eta \leq \varepsilon m }] \sim \frac{8}{\pi m \log m }\sum_{i=1}^m  (\log m)  \overline X_{\rho_{i-1}}=\frac{8}{\pi m}\sum_{i=1}^m    \overline X_{\rho_{i-1}} .
 \end{align*}
 Then using that $(X_{\rho_{i}},i)$ is positive recurrent, the fact that its  invariant probability measure $(\pi^{\dag}(x),x)$ admits a first moment (see Lemma \ref{lemrec2}) and the Birkhoff ergodic Theorem, we obtain for large $m$, that in probability 
\begin{align*}
\frac{1}{m \log m }\sum_{i=1}^m \E_{X_{\rho_{i-1}}}[\eta \un_{\eta \leq \varepsilon m}]\sim  \frac{8}{\pi m}\sum_{i=1}^m    \overline X_{\rho_{i-1}}  \overset{\P.a.s.}{ \rightarrow }  \frac{8}{\pi}  \sum_{x} \overline x  \pi^{\dag}(x). 
\end{align*}
We deduce from that, in probability  
\begin{align*}
\lim_{m \rightarrow +\infty}\frac{1}{m \log m} \sum_{i=1}^m \E_{X_{\rho_{i-1}}}[\eta \un_{\eta \leq \varepsilon m}]= \frac{8}{\pi} \sum_{x} \overline x \pi^{\dag}(x). 
 \end{align*}
We are left to prove \eqref{neglpart},  like in Lemma \ref{lem3.2}  strong Markov property yields 
\begin{align}
& \E\Big[\Big( \Sigma_1-\sum_{i=1}^m\E_{X_{\rho_{i-1}}}[\eta \un_{\eta \leq \varepsilon m}]\Big)^2  \Big]  
\leq \sum_{i=1}^m  \E\left[ \E_{X_{\rho_{i-1}}}\left[\eta^2 \un_{\eta \leq \varepsilon m} \right] \right] \label{55}.
\end{align} 
Again using Lemma \ref{MarcheSimpleTemps}
\begin{align*}
 \E_{X_{\rho_{i-1}}}\left[\eta^2 \un_{\eta \leq \varepsilon m})\right]
 &
 \leq C_+ \sum_{k \leq \varepsilon m} k \P_{X_{\rho_{i-1}}}(\eta>k)\\ 
&\leq C_+ (\log m)^2+ \sum_{\log m \leq k \leq \varepsilon m} k \P_{X_{\rho_{i-1}}}(\eta>k)(\un_{\oX_{\rho_{i-1}} \leq k^{1/2-\delta}}+\un_{\oX_{\rho_{i-1}} > k^{1/2-\delta}}  )   \\ 
 & \leq  C_+ (\log m)^2+ C_+ \sum_{\log m \leq k \leq \varepsilon m}  \overline X_{\rho_{i-1}}\un_{\oX_{\rho_{i-1}} \leq k^{1/2-\delta}} + \varepsilon m F(0,0). 
 \end{align*}
 as by Lemma \ref{lem410} $\E[\overline X_{\rho_{i-1}}]\leq  c $ for any $i$, in particular, the mean of the sum in the above inequality is bounded by a constant times $\varepsilon m$. Thanks to our previous computations, we have that $\varepsilon m\E[F(0,0)]=o(m)$. 
Therefore  $\E\left[ \E_{X_{\rho_{i-1}}}\left[\eta^2 \un_{\eta \leq \varepsilon m}\right]\right] \leq C_+ m  \varepsilon$, so the second moment  \eqref{55} is smaller that $C_+ \varepsilon m^2$. Finally using (second moment) Markov inequality 
 in the probability in \eqref{neglpart}, yields \eqref{neglpart}. 
\end{proof}

\noindent The proof of Proposition \ref{prop31} then writes: by Corollary \ref{cor1} and Lemma \ref{lem3.3bis}, in probability \begin{align*} 
\lim_{k \rightarrow +\infty} \frac{\rho_k}{k \log k} = \lim_{k \rightarrow +\infty} \left( \frac{\sum_{i=1}^{k} \eta_i-\rho_{i-1}}{k \log k}+\frac{\sum_{i=1}^k \rho_i-\eta_{i}}{k \log k}  \right)= \lim_{k \rightarrow +\infty} \frac{\sum_{i=1}^k \eta_i-\rho_{i-1}}{k \log k}   =\frac{8}{\pi} \sum_{x} \overline x \pi^{\dag}(x),  
\end{align*}
 which yields the result by definition of $N_m$.

 \subsection{Proof of Theorem \ref{mainth}}
 By Propositions \ref{prop31}, for any $\varepsilon>0$ with probability converging to one $ \frac{\log n}{n}  \sum_{i=1}^{\lfloor  (1-\varepsilon)\cun n/\log n \rfloor }f( \mathscr{B}_i) \leq \frac{\log n}{n}  \sum_{i=1}^{N_n}f( \mathscr{B}_i) \leq \frac{\log n}{n}  \sum_{i=1}^{\lceil (1+\varepsilon) \cun n/\log n \rceil }f( \mathscr{B}_i) $ then, as $\lim_{t\rightarrow+\infty}\nicefrac{\lfloor (1-\varepsilon)\cun  t\rfloor}{t}=(1-\varepsilon)\cun  $, Proposition  \ref{prop1} gives  for the lower bound
 \begin{align*}
 \frac{\log n}{n}  \sum_{i=1}^{\lfloor \cun (1-\varepsilon) n/\log n \rfloor } f( \mathscr{B}_i)
\overset{\P}{ \rightarrow } (1-\varepsilon)\cun \sum_{x \in K^c} \pi^*(x)\E_x[f(\mathscr B_0)]=\cn^f(1-\varepsilon) .
\end{align*}
A similar result is true for the upper bound, taking the limit when $\varepsilon \rightarrow 0$ yields $\frac{\log n}{n}  \sum_{i=1}^{N_n}f( \mathscr{B}_i) \overset{\P}{ \rightarrow } \cn^f$. So we are left to prove that $\frac{\log n}{n}f( \mathscr{B}_n^* ) \overset{\P}{ \rightarrow } 0$. 
Let us give an upper bound for $\P(\frac{\log n}{n}f( \mathscr{B}_n^* )>n^{-\varepsilon/2})$: 
using successively Lemma \ref{simp3.1} implying $\P(\cup_{i=1}^n\{\oX_{\eta_{i}}> n^{1/2+2 \varepsilon}\})\le C_+ {n^{-4 \varepsilon}}$, and the increasing property of $f$ giving that $f( \mathscr{B}_n^* )\le \max_{1\le i\le n}f(\mathscr B_i)$: 
\begin{align}
\P\Big(\frac{\log n}{n}f( \mathscr{B}_n^* )>n^{-\varepsilon/2}\Big)
&\leq \P\Big(\frac{\log n}{n}f( \mathscr{B}_n^* )>n^{-\varepsilon/2},\cap_{i=1}^n\{\oX_{\eta_{i}}\leq n^{1/2+2 \varepsilon}\} \Big)+C_+ {n^{-4 \varepsilon}} \nonumber \\
& \leq \P\Big(\bigcup_{i=1}^n\Big\{\frac{\log n}{n}f(\mathscr{B}_i)>n^{-\varepsilon/2}, \oX_{\eta_{i}}\leq n^{1/2+2 \varepsilon} \Big \}\Big)+C_+ {n^{-4 \varepsilon}}\nonumber\\
&\leq n\max_{1\le i\le n}\P\Big(\frac{\log n}{n}f(\mathscr{B}_i)>n^{-\varepsilon/2}, \oX_{\eta_{i}}\leq n^{1/2+2 \varepsilon}\Big)+C_+ {n^{-4 \varepsilon}}\label{leq}.
\end{align}
To deal with this probability we first compare each random variable $f(\mathscr B_i)$ with $\E_{X_{\eta_i}}[f(\mathscr{B}_0)]$. In the one hand by strong Markov property, Tchebychev inequality and finally condition \eqref{cond1}
\begin{align}
& \P\Big( |f(\mathscr{B}_i)-\E_{X_{\eta_i}}[f(\mathscr{B}_0)]|>\frac{n^{1-\varepsilon}}{\log n}, \oX_{\eta_i} \leq n^{1/2+2 \varepsilon} \Big) \nonumber \\
&=\E\Big[\mathds{1}_{X_{\eta_i}\le n^{1/2+2\varepsilon}} \P_{X_{\eta_i}}\Big( |f(\mathscr{B}_0)-\E_{X_{\eta_i}}[f(\mathscr{B}_0)]|>\frac{n^{1-\varepsilon}}{\log n} \Big)\Big]\le \frac{(\log n)^2}{n^{2-2\varepsilon}}\E\Big[\mathds{1}_{X_{\eta_i}\le n^{1/2+2\varepsilon}} \mathbb{V}\mathrm{ar}_{X_{\eta_i}}[f(\mathscr B_0)]\Big]\nonumber\\
&\le C_2 \frac{(\log n)^2}{n^{2-2\varepsilon}}\E\Big[\mathds{1}_{X_{\eta_i}\le n^{1/2+2\varepsilon}}X_{\eta_i}^{2-\delta} \Big]
\leq C_2 \frac{ (\log n)^2}{n^{1-6\varepsilon+\nicefrac{\delta}{2}+2\delta\varepsilon}}\leq \frac{1}{n^{1-6 \varepsilon +\delta/2}},\label{ppro1}
\end{align}
recall indeed that $\delta$ introduced in condition \eqref{cond1} is given and $\varepsilon$ can be chosen as small as we want so in particular above probability converges to zero as long as $\delta/2 >6 \varepsilon$. On the other hand we have to control the sequence $(\E_{X_{\eta_i}}[f(\mathscr{B}_0)],i \leq n)$: using condition \eqref{cond2} 
and lemma \ref{newlem51} for $1<\beta<2$
\begin{align}
 \P\Big( \E_{X_{\eta_i}}[f(\mathscr{B}_0)]>\frac{n^{1-\varepsilon}}{ \log n}\Big)
 \leq   \P\Big( C_2 X_{\eta_i} >\frac{n^{1-\varepsilon}}{\log n}\Big) & \leq C\frac{(\log n)^{\beta}}{n^{\beta(1-\varepsilon)}} \E[X_{\eta_i}^{\beta}] \leq C\frac{(\log n)^{\beta}}{n^{\beta(1-\varepsilon)}}\label{ppro2}.
\end{align}
 Taking $\beta(1-\varepsilon)>1$ ensures that $n  \max_{i \leq n}   \P\Big( \E_{X_{\eta_i}}[f(\mathscr{B}_0)]>\frac{n^{1-\varepsilon}}{ \log n}\Big)$ converges to 0. Collecting \eqref{leq},\eqref{ppro1} and  \eqref{ppro2} 
 implies that $\lim_{n \rightarrow +\infty} \P\Big(\frac{\log n}{n}f( \mathscr{B}_n^* )>n^{-\varepsilon/2}\Big)=0$, which leads to the desired result. \hfill $\square$

 \section{Reversibility and  technical Lemmata for trajectories on the axis \label{sec4}}

In this section  we prove technical estimates for the part of the trajectory restricted to the axis, note that we  do not  need the condition $\alpha>3$ here, in fact most of the time $\alpha>1$ is enough so we mention the condition for $\alpha$ for each statement. For typographical simplicity, for any $x$ and $y$ in $K^c$, let $\P(x\rightarrow y)$ be the probability of the shortest path on $K^c$ to join $y$ from $x$. For instance, on figure \ref{shortpath}:
\begin{align*} \red{\P(x\rightarrow y)=\prod_{i=1}^{\ox} \left(1-\frac{3}{4i^\alpha}\right)\prod_{i=1}^{\oy-1}\frac{1}{4i^\alpha}.}
\end{align*}
Note that the expression of $\P(x\rightarrow y)$ has a useful consequence to compute probability of the form $\P_{\cdot}(X_{\rho}=\cdot)$. 
For example, one can see that there exists $0<C<1$ such that:
\begin{equation}\label{infcone}
    \P_x(X_\rho=(1,1))>C,\, \forall x\in K^c.
\end{equation}
Indeed, as $\alpha>1$, $\P(x \rightarrow (0,0))>\prod_{i=1}^{\infty}\left(1-\frac{3}{4i^\alpha}\right) :=a>0 $. Then:
\begin{align*}
\P_{x}\left(X_\rho=(1,1)\right)\ge  \P\left(x\rightarrow (0,0)\right)p((0,0),(0,1))p((0,1),(1,1))=\frac{a}{4^2}=:C>0.
\end{align*}

The first statement below treats about the reversibility of the random walk on $K^c$.

\begin{figure}
\includegraphics[scale=0.5]{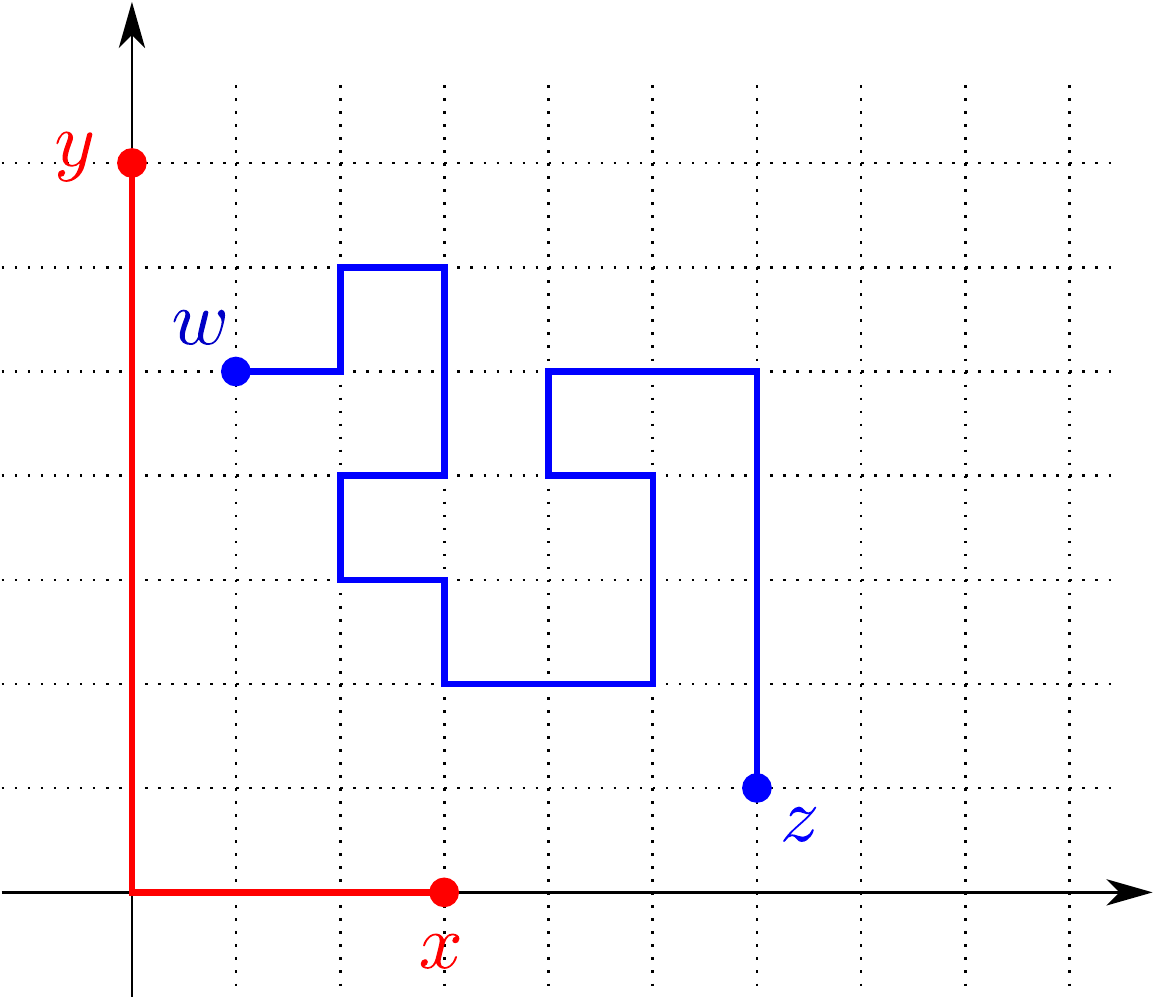}
\caption{}
\label{shortpath}
\end{figure}

\begin{lem} \label{lem2}
For all $x, y\in K^c$ such that $x+e_i, y+e_j\in K$:
\begin{align*}
\P_{x}(X_\rho=y+e_j)=\frac{\P(x\rightarrow y)}{\P(y\rightarrow x)}\frac{\ox^\alpha}{\oy^\alpha}\P_{y}(X_\rho=x+e_i).
\end{align*}
\end{lem}
\begin{proof}
First we prove that for all $x,y\in K^c$ and all $n\in\mathbb N^*$:
\begin{equation}\label{inversion}
\P_x(X_n=y, n<\rho)=\frac{\P(x\rightarrow y)}{\P(y\rightarrow x)}\P_y(X_n=x, n<\rho).
\end{equation}
This fact is a simple consequence of the reversibility of random walk ${\bf X}$. 
To prove \eqref{inversion}, take a path $\Gamma$ from $x$ to $y$ on $K^c$ of length $n$, $\Gamma:=(x_0=x,\,x_1,\,\dots,\, x_{n-1}, x_n=y)$. Its probability is
\[\P_x(X_1=x_1,\dots, X_n=y)=\prod_{i=0}^{n-1}p(x_i,x_{i+1})=\P(x\rightarrow y) A_{\Gamma}\]
and note that $A_{\Gamma}$ is a product such that if $p(x_i, x_{i+1})$ appears in $A_{\Gamma}$ there is also necessarily $j\ne i$, such that $p(x_j, x_{j+1})=p(x_{i+1}, x_i)$.\\ 
If we reverse the path (taking $i\rightarrow n-i$), we obtain similarly
\[\P_y(X_1=x_{n-1},\dots, X_{n-1}=x_1,X_n=x)=\prod_{i=0}^{n-1}p(x_{i+1},x_i)=\P(y\rightarrow x)A_{\Gamma}\]
as the reversion does not change the value of $A_{\Gamma}$.\\
As a result summing on all path $\Gamma$ of length $n$ from $x$ to $y$:
\begin{align*}
\P_x(X_n=y, n<\rho)&=\sum_{\Gamma}\prod_{i=0}^{n-1}p(x_i,x_{i+1})=\P(x\rightarrow y)\sum_{\Gamma} A_{\Gamma}=\frac{\P(x\rightarrow y)}{\P(y\rightarrow x)}\sum_{\Gamma} \P(y\rightarrow x)A_{\Gamma} \\
&=\frac{\P(x\rightarrow y)}{\P(y\rightarrow x)}\sum_{\Gamma} \prod_{i=0}^{n-1}p(x_{i+1},x_i)=\frac{\P(x\rightarrow y)}{\P(y\rightarrow x)}\P_y(X_n=x, n<\rho).
\end{align*}
Now, the result of the lemma follows taking $x, y\in K^c$ such that $x+e_i, y+e_j\in K$:
\begin{align*}
\P_{x}(X_\rho=y+e_j)&=\sum_{n\ge 0}\P_{x}(X_n=y,n<\rho, X_{n+1}=y+e_j)= \sum_{n\ge 0}\P_{x}(X_n=y,n<\rho)p(y,y+e_j)\\
&=\frac{\P(x\rightarrow y)}{\P(y\rightarrow x)}p(y,y+e_j) \sum_{n\ge 0}\P_{y}(X_n=x,n<\rho)\\
&=\frac{\P(x\rightarrow y)}{\P(y\rightarrow x)}\frac{p(y,y+e_j)}{p(x,x+e_i)}\P_{y}(X_\rho=x+e_i)=\frac{\P(x\rightarrow y)}{\P(y\rightarrow x)}\frac{\ox^\alpha}{\oy^\alpha}\P_{y}(X_\rho=x+e_i).
\end{align*}
\end{proof}

\begin{rem}\label{reminver}
There is a counterpart of the precedent result on $K$: For all $w,z\in K$ and all $n\in\mathbb N^*$ such that $w+e_i, z+e_j\in K^c$
\begin{equation}\label{inversioncool}
\P_{w}\left(X_\eta=z+e_j\right)=\P_{z}\left(X_\eta=w+e_i\right). 
\end{equation}
The proof of \eqref{inversioncool} is very similar to \eqref{inversion} and is left to the reader (see Figure \ref{shortpath}).
\end{rem}

The lemma below gives an asymptotic of the distribution of the exit coordinate from the axis.  We  use the following notation: for any $x\in \mathbb Z^2$, let $(T_x^k)_{k\ge 0}$ be the sequence defined by $T_x^0=0$ and for all $k\in\mathbb N^*$:
\[T_x^k=\inf\lbrace k>T_x^{k-1}, X_k=x\rbrace,\] 
for simplicity we write $T_x$ instead of $T^1_x$. 
\begin{lem}\label{exitK} Assume $\alpha>1$, there exists $c_+ > c_->0 $ such that for all $i>1$
\begin{equation} \label{simpler?}
1+c_- i^{-\alpha} \le 4i^\alpha \P_{(0,i)}(X_\rho=(1,i))\le 1+c_+ i^{-\alpha}.
\end{equation} 
\end{lem}
\begin{proof}
Using the strong Markov property:
\begin{align}
\P_{(0,i)}(X_\rho=(1,i))
&=\sum_{k\ge0}\P_{(0,i)}\left(T_{(0,i)}<\rho\right)^k\P_{(0,i)}\left(\rho<T_{(0,i)},X_\rho=(1,i)\right)\nonumber\\
&=p((0,i),(1,i))\sum_{k\ge0}\P_{(0,i)}\left(T_{(0,i)}<\rho\right)^k \nonumber \\
& = \frac{1}{4 i^{\alpha}}\frac{1}{1-\P_{(0,i)}\left(T_{(0,i)}<\rho\right)}=:\frac{1}{4 i^{\alpha}} \frac{1}{1-h(i)}. \label{all}
\end{align}
In order to obtain a lower bound for $h(i)$, we apply the Markov property several times: 
\begin{align}
h(i)&=\frac{1}{4i^\alpha}h(i+1)+\left(1-\frac{3}{4i^\alpha}\right)h(i-1)\label{upper1}\\
&=\frac{1}{4i^\alpha}\left(1-\frac{3}{4(i+1)^\alpha}+\frac{1}{4(i+1)^\alpha}h(i+2)\right)+\left(1-\frac{3}{4i^\alpha}\right)\left(\frac{1}{4(i-1)^\alpha}+\left(1-\frac{3}{4(i-1)^\alpha}\right)h(i-2)\right)\label{bound}\\
&\ge\frac{1}{4i^\alpha}\left(1-\frac{3}{4(i+1)^\alpha}\right)+\left(1-\frac{3}{4i^\alpha}\right)\frac{1}{4(i-1)^\alpha} \geq  \frac{c_-}{i^\alpha} \nonumber. 
\end{align}
The upper bound is also obtained from \eqref{upper1}: first $h(i+1)$ is simply bounded from above by $1-\frac{1}{2(i+1)^\alpha}$ using \eqref{bound}. We treat $h(i-1)$ separately with a similar reasoning as the one to obtain \eqref{all} and taking a particular trajectory:
\begin{align*}
h(i-1)&=\frac{p((0,i-1),(0,i))}{1-\P_{(0,i-1)}\left(T_{(0,i-1)}<T_{(0,i)}\wedge \rho\right)}=\frac{1}{4(i-1)^\alpha\P_{(0,i-1)}\left(T_{(0,i-1)}>T_{(0,i)}\wedge \rho\right)}\nonumber\\
&\le \frac{1}{4(i-1)^\alpha\P_{(0,i-1)}\left(T_{(0,i-1)}> \rho, X_\rho\in\lbrace (-1,1);(1,1)\rbrace\right)}\nonumber\\
&\le\frac{1}{4(i-1)^\alpha}\frac{1}{\frac{1}{2}\prod_{k=2}^{i-1}\left(1-\frac{3}{4k^\alpha}\right)}, 
\end{align*}
 then 
\[h(i)\le \frac{1}{4i^\alpha}\left(1-\frac{1}{2(i+1)^\alpha}\right)+\frac{1}{4(i-1)^\alpha}\frac{1}{\frac{1}{2}\prod_{k=2}^{i-1}\left(1-\frac{3}{4k^\alpha}\right)},
\]
and as $\alpha>1$, $\prod_{k=2}^{i-1}\left(1-\frac{3}{4k^\alpha}\right)$ is strictly positive constant so we get the upper bound. 
\end{proof}
\noindent The following Corollary is a consequence of Lemma \ref{exitK}. Recall that for $i>0$,  $L_{(0,i)}=\{z= (0,z_2),z_2 \geq i \}.$

\begin{cor} \label{cor4.6}
For $\alpha>1$, all $z\in K^c$ and $i>0$:
\begin{equation}\label{easypeasy}
\P(z\rightarrow(0,i))\le4i^\alpha\P_z(X_\rho=(1,i))\le  1+c_+ i^{-\alpha}.
\end{equation}
Moreover, 
\begin{enumerate}
\item if $z\in L_{(0,i)}$: 
\begin{equation}
\lim_{i\rightarrow+\infty}4i^\alpha\P_z(X_\rho=(1,i))=1 ;
\end{equation}
\item there exists $C^\prime>0$ such that for all $z\notin L_{(0,i)}$:
\begin{equation}
\P_{z}(X_{\rho}=(1,i))\le \frac{C^\prime}{i^{2\alpha}}.
\end{equation}
\item there exists $0<\tilde C<1$ such that for all $z\in K^c$, $\P_z(T_z<\rho)\le \tilde C.$
\end{enumerate}
\end{cor}

\begin{proof}
Using the strong Markov property:
\begin{align}\label{easyMarkov}
    \P_z(X_\rho=(1,i))=\P_z(T_{(0,i)}<\rho, X_\rho=(1,i))=\P_z(T_{(0,i)}<\rho)\P_{(0,i)}( X_\rho=(1,i)),
\end{align}
which implies \eqref{easypeasy} using Lemma \ref{exitK} and the fact that $\P(z\rightarrow(0,i)) \le \P_z(T_{(0,i)}<\rho)$.
\begin{enumerate}
\item This formula is direct as 
\[\P(z\rightarrow(0,i)) =\prod_{k=i+1}^{z}\left(1-\frac{3}{4k^\alpha}\right)\ge \prod_{k=i+1}^{\infty}\left(1-\frac{3}{4k^\alpha}\right)=:R_i\underset{i\rightarrow+\infty }{\longrightarrow}1,\]
as $R_i$ is the rest of a convergent infinite product.
\item
With \eqref{easyMarkov}, the inequality $\P_z(T_{(0,i)}<\rho)\le \P_{{(0,i-1)}}(T_{(0,i)}<\rho)$ and an obvious symmetry:
\begin{align*}
\P_{{(0,i-1)}}(X_{\rho}=(1,i-1), T_{(0,i)}>\rho)
\le \P_{{(0,i-1)}}(X_{\rho}=(1,i-1)),
\end{align*}
and \eqref{easypeasy} implies the existence of $C^\prime>0$ such that for all positive $i$, $C^\prime\ge 1+c_+ i^{-\alpha}$.
\item We can take $z=(0,i)$ for $i\ge 1$ without loss of generality. Using \eqref{upper1}, the same symmetry as the previous point and \eqref{easypeasy}:
\begin{align*}
\P_{(0,i)}(T_{(0,i)}<\rho)\le \frac{1}{4i^\alpha}+\P_{(0,i-1)}(T_{(0,i)}<\rho) \le  \frac{1}{4i^\alpha}+\P_{(0,i-1)}(X_{\rho}=(1,i-1))\le \frac{1}{4i^\alpha}\left(2+\frac{c_+}{i^\alpha}\right).
\end{align*}
As a result, we can find $N>0$, such that for all $i\ge N$, $\P_{(0,i)}(T_{(0,i)}<\rho)\le \frac{1}{2}$, and taking $\tilde C:=\max(\max_{0\le i<N} \P_{(0,i)}(T_{(0,i)}<\rho) ,\frac{1}{2})$, we have the claimed result.
\end{enumerate}
\end{proof}

\begin{cor}\label{mean}
Let $\beta>0$ such that $\alpha-\beta>1$, there exists $M>0$ such that: 
\begin{equation}
\forall z\in K^c, \E_z\left[\overline{X}_\rho^\beta\right]\le M. 
\end{equation}
\end{cor}

\begin{proof}

Without loss of generality, we can assume that $z=(0,i)$ with $i\ge1$. Using the symmetry of the problem and \eqref{simpler?}, for $j>2$:
\begin{align*}
\P_{(0,i)}\left(\oX_{\rho}=j\right)&=2\P_{(0,i)}\left(X_{\rho}=(1,j)\right)+6\P_{(0,i)}\left(X_{\rho}=(j,1)\right)\\
&=2\P_{(0,i)}\left(T_{(0,j)}< \rho\right)\P_{(0,j)}\left(X_{\rho}=(1,j)\right)+6\P_{(0,i)}\left(T_{(j,0)}< \rho\right)\P_{(j,0)}\left(X_{\rho}=(j,1)\right)\\
&\le 8\P_{(0,j)}\left(X_{\rho}=(1,j)\right)\le C_+j^{-\alpha}.
\end{align*}
Consequently: 
\begin{align*}
 \E_{(0,i)}\left[\overline{X}_\rho^\beta\right]&=\sum_{j\ge 1}\P_{(0,i)}\left(\oX_{\rho}=j\right)j^{\beta}=\P_{(0,i)}\left(\oX_{\rho}=1\right)+\sum_{j\ge 2}\P_{(0,i)}\left(\oX_{\rho}=j\right)j^{\beta}\le 1+C_+\sum_{j\ge 2}j^{\beta-\alpha}=M.
 \end{align*}
\end{proof}
\noindent The last Lemma is a simple consequence of Corollary \ref{mean}  and Corollary \ref{cor4.6}
\begin{lem}\label{lem410} Assume $\alpha>2$, there exists $c$ and $c'$  such that for any $i$ and $x \in K$  
\begin{equation}
\E[\oX_{\rho_i}] \leq c \textrm{ and } \P(X_{\rho_i}=x) \leq c'/\ox^{\alpha}.
\end{equation} 
\end{lem}
\begin{proof}
According to Corollary \ref{mean}, there exists $M>0$ such that for any 
$\E[\overline{X}_{\rho_i}] =\E[\E_{X_{\eta_i}}[\oX_{\rho}]]\le M$.\\
Similarly, writing $\P(X_{\rho_i}=x)=\E[\P_{X_{\eta_i}}(X_{\rho}=x)]$, one can obtain the second inequality with \eqref{easypeasy}.
\end{proof}

\begin{lem} \label{lem6.1} Assume $\alpha>1$.  For any $a,r>0$ and $m$ large enough 
\begin{align}
\P_{(0,a)}(\rho > m ) \leq m^{-r}. \label{eq34}
\end{align}
\end{lem}

\begin{proof}
For all $k\in \N^*$, let $T^*_k=\inf\lbrace n\ge 0, \overline{X}_n=k\rbrace$. 
As $m$ goes to infinity, we can assume without loss of generality that $2a<\log m$. Then, one can write (we have chosen to not indicate the integer parts for typographical simplicity)
: 
\[\P_{(0,a)}(\rho > m )=\P_{(0,a)}(\rho > m,  T^*_{ \log m} < \rho )+\P_{(0,a)}(T^*_{ \log m}  > \rho>m  ):=(I)+(II).\] 
For part $(I)$ by symmetry, the fact that $\P_{(0,a)}(T_{(0,\log m)}<\rho)>\P_{(0,a)}(T_{(0,-\log m)}<\rho)$ and Lemma \ref{lem2}, we obtain:
\begin{align*}
(I) &\le 4 \P_{(0,a)}( T_{(0, \log m)} < \rho ) =4 \frac{\P_{(0,a)}\left(X_\rho=(1,\log m)\right)}{\P_{(0,\log m)}\left(X_\rho=(1,\log m)\right)}\le 16 (\log m)^\alpha \P_{(0,a)}\left(X_\rho=(1,\log m)\right)\\
&  = 16 a^{\alpha }\frac{\P((0,a)\rightarrow(0,\log m))}{\P((0,\log m)\rightarrow (0,a))} \le 16a^{\alpha} \frac{\prod_{i=a}^{\log m-1}\frac{1}{4i^\alpha} }{\prod_{i=1}^{\infty }\left(1-\frac{3}{4i^\alpha}\right)}= 4C\prod_{i=a+1}^{\log m-1}\frac{1}{4i^\alpha}\le 4C \prod_{i=\frac{\log m}{2}}^{\log m-1}\frac{1}{4\left(\frac{\log m}{2}\right)^\alpha}\\
& \leq 4C\left(4\left(\frac{\log m}{2}\right)^\alpha\right)^{-\frac{\log m}{2}}, 
\end{align*}
which implies that for all $r>0$, ($I$) is bounded from above by $m^{-r}$ for $m$ large enough.\\ 
 Otherwise, for second term $(II)$ 
 \begin{align*}
(II)&=\P_{(0,a)}( T^*_{ \log m} > \rho>m  )=\P_{(0,a)}\left(\sum_{z\in K^c,|z| \leq  \log m} \mathscr{L}(z,\rho)>m\right)\\
 &\le \P_{(0,a)}\left(\exists z\in K^c,|z|\le \log m, \mathscr L(z,\rho)  > \frac{m}{4\log m+1}\right)\le \sum_{z\in K^c, |z|\le \log m}\P_{(0,a)}\left(\mathscr{L}(z,\rho)> \frac{m}{4\log m+1}\right).
 \end{align*} 
 Note that for all $k \ge 2$ and all $z\in K^c$, using the third point of Corollary \ref{cor4.6}:
 \begin{align*}
 \P_{(0,a)}\left(\mathscr{L}(z,\rho)> k\right)=\P_{(0,a)}\left(T_{z}<\rho \right) \P_{z}\left(T_{z}<\rho\right)^k\le \P_{z}\left(T_{z}<\rho\right)^k\le \tilde C^k.
 \end{align*}
Consequently:
\begin{align*}
(II)&\le \sum_{z\in K^c, |z| \leq  \log m}\tilde C^{\frac{m}{4\log m}}=(4\log m+1)\tilde C^{\frac{m}{4\log m}},
 \end{align*} 
 so, for all $r>0$ , ($II$) is also bounded from above by $m^{-r}$ for $m$ large enough.
\end{proof}

\noindent In the following Lemma we obtain asymptotic (with respect to the coordinate of the starting point) of the exit time from the axis $\rho$ and its second order. 

\begin{lem} \label{lem6.2} Assume $\alpha>1$, let $0< \beta \leq 2$ and $0<\varepsilon<1$. Then, for large  $i$ 
\begin{equation}
|\E_{(0,i)}[\rho^{\beta}]-i^{\beta}|=O(i^{\beta -\varepsilon}).
\end{equation} 
\end{lem}

\begin{proof}
For the lower bound, just note that
\begin{align*}
\E_{(0,i)}\left[\rho^\beta\right]&\ge \E_{(0,i)}\left[\rho^\beta\mathds{1}_{\rho\ge i-i^{1-\varepsilon}}\right]\ge (i-i^{1-\varepsilon})^\beta\P_{(0,i)}\left(\rho\ge i-i^{1-\varepsilon}\right)\\
&= (i-i^{1-\varepsilon})^\beta\left(1-\P_{(0,i)}\left(\rho< i-i^{1-\varepsilon}\right)\right)=:L(i)
\end{align*}
\red{
and with \eqref{easypeasy}:
\begin{align*}
\P_{(0,i)}\left(\rho< i-i^{1-\varepsilon}\right)&\le\sum_{k\ge i^{1-\varepsilon}}\P_{(0,i)}\left(X_{\rho}=(1,k)\right) \le \sum_{k\ge i^{1-\varepsilon}}\P_{(0,k)}\left(X_{\rho}=(1,k)\right) \\
& \le  \sum_{k\ge i^{1-\varepsilon}}\frac{1}{4 k^{\alpha}}\left(1+ \frac{c_+}{k^{\alpha}}\right) \le C_+ i^{-(\alpha-1)(1-\varepsilon)}.
\end{align*}
This implies that for large $i$,  
$L(i)=(i-i^{1-\varepsilon})^{\beta}(1+o(1))=i^{\beta}+O(i^{\beta-\varepsilon})$.
For the upper bound, 
\begin{align*}
\E_{(0,i)}\left[\rho^\beta\right]& =\E_{(0,i)}\left[\rho^\beta\left(\mathds{1}_{\rho< i+i^{1-\varepsilon}}+\mathds{1}_{\rho\ge i+i^{1-\varepsilon}}\right)\right] \\
& \le \left(i+i^{1-\varepsilon}\right)^{\beta}+\E_{(0,i)}\left[\rho^\beta\mathds{1}_{\rho\ge i+i^{1-\varepsilon}}\right]=:  \left(i+i^{1-\varepsilon}\right)^{\beta}+U(i).
\end{align*}
Then, using \eqref{eq34},  with $r>\beta+1$ and $i$ large enough:
\begin{align*}
U(i)=\sum_{k\geq i^{1-\varepsilon}}(i+k)^{\beta}\P_{(0,i)}\left(\rho=i+k\right)\le \sum_{k\geq i^{1-\varepsilon}}\frac{(i+k)^\beta}{k^r}=o(1). 
\end{align*}
This finishes the proof.
}\end{proof}

\section{Technical lemmata for trajectories on the cone \label{lecone}}

Recall the definition of $\eta= \inf \{k>0, X_k \in K\}$. In this section we obtain a (uniform) local limit result for $X_{\eta}$ (Lemma \ref{MarcheSimple}) as well as a uniform tail for $\eta$ (Lemma \ref{MarcheSimpleTemps}).
\begin{lem} \label{MarcheSimple} 
 For any $\delta>0$, uniformly in $y \leq x^{1-\delta}$:
\begin{align*}
& \lim_{x \rightarrow +\infty} \frac{x^3}{y}\P_{(1,x)}(X_{\eta}=(y,0))= \lim_{x \rightarrow +\infty} \frac{x^3}{y}\P_{(y,1)} (X_{\eta}=(0,x))=\frac{16}{\pi}. \\
&
\lim_{x \rightarrow +\infty} \frac{x^3}{y}\P_{(1,x)}(X_{\eta}=(0,y)) =\lim_{x \rightarrow +\infty} \frac{x^3}{y}\P_{(1,y)} (X_{\eta}=(0,x))  =\frac{16}{\pi}.
\end{align*}
\end{lem}

\begin{proof}
Note that the first equality on both above lines comes from the symmetry of the distribution of the simple random walk on the cones (see Remark \ref{reminver}). \\
For any sequence $(Y_n,n\ge1)$ of real random variables, introduce $\underline Y_{n}:=\inf_{1\le k\le n}Y_k$.\\
\underline{We start with  $\P_{(1,x)}(X_{\eta}=(y,0))$}. 
First, writing $X_n=(X_n^1,X_n^2)$, we easily see that:   
\begin{align*}
\P_{(1,x)}(X_{\eta}=(y,0)) & =\sum_{k \geq x} \P_{(1,x)}(X_{k}=(y,0), \eta=k )=  \sum_{k \geq x} \P_{(1,x)}(\underline X_{k}^1>0, X_{k}^1=y, \underline X_{k-1}^2>0,  X_{k}^2=0)\\
&=\frac{1}{4}\sum_{k\ge x}\P_{(1,x)}(\underline X_{k-1}^1>0, X_{k-1}^1=y, \underline X_{k-1}^2>0,  X_{k-1}^2=1),
\end{align*}
as the $k$-th step is necessarily vertical, more precisely  $X_{k-1}^2=1$ and $X_{k}^2=0$.\\ 
If $\mathscr H_j^{k}$ is the event {\it{$\lbrace $among the first $k-1$ steps, there is exactly $j$ horizontal ones$\rbrace$}} and if $Z$ is the symmetric random walk on $\mathbb Z$, one can write:
\begin{align}
\ell_k^j:=& \P_{(1,x)}(\underline X_{k-1}^1>0, X_{k-1}^1=y, \underline X_{k-1}^2>0,  X_{k-1}^2=1|\mathscr H_j^{k})  \nonumber \\
=&\P_{1}(\underline Z_{j}>0, Z_{j}=y)\P_x( \underline Z_{k-j-1}>0,  Z_{k-j-1}=1). \label{lkj}
\end{align}
Thus, for $0<\varepsilon<1$:
\begin{align}
\P_{(1,x)}(X_{\eta}=(y,0) )&=\frac{1}{4}\sum_{k\ge x}\sum_{j=y-1}^{k-1-x}\P(\mathscr H_j^{k})\P_{1}(\underline Z_{j}>0, Z_{j}=y)\P_x( \underline Z_{k-j-1}>0,  Z_{k-j-1}=1)\nonumber\\
&=\frac{1}{4}\sum_{k\ge x}\sum_{j\in B_{k,\varepsilon}}\P(\mathscr H_j^{k})\ell_k^j+\frac{1}{4}\sum_{k\ge x}\sum_{j\in B^c_{k,\varepsilon}}\P(\mathscr H_j^{k})\ell_k^j=:\Sigma_1+\Sigma_2, \label{eqtout}
\end{align}
where $B_{k,\varepsilon}:=\left[\nicefrac{(k-1)(1-\varepsilon)}{2},\nicefrac{(k-1)(1+\varepsilon)}{2}\right]$ and $B_{k,\varepsilon}^c$ its  complementary in $\{  y-1, \cdots, k-1-x \}$.\\
First note that according to \eqref{BE2}, for $j$ in $B_{k,\varepsilon}$ $\P(\mathscr H_j^{k}) \leq e^{-\frac{\varepsilon^2 k}{6} }$, implying:
\begin{align}\label{neglig1}
\Sigma_2\le \sum_{k\ge x}e^{-\frac{\varepsilon^2 k}{6}}\sum_{j\in B_{k,\varepsilon}^c} \ell_k^j
\le \sum_{k \ge x} k e^{-\frac{\varepsilon^2 k}{6}}\le \int_{x}^{\infty}te^{-\frac{\varepsilon^2 t}{6}}\mathrm{d}t=e^{-\frac{\varepsilon^2 x}{6}}\frac{6}{\varepsilon^2}\Big(x+\frac{6}{\varepsilon^2} \Big)
\end{align}
and as a result $\lim_{\varepsilon\rightarrow0}\lim_{x\rightarrow+\infty} \frac{x^3}{y}\Sigma_2=0.$ In view of what we want to prove, we only consider $j\in B_{k,\varepsilon}$ in the following and we write:
\[\Sigma_1=\frac{1}{4}\Big(\sum_{x\le k< \varepsilon x^2}+\sum_{\varepsilon x^2\le k\le x^2/\varepsilon}+\sum_{k>x^2/\varepsilon}\Big)\sum_{j\in B_{k,\varepsilon}}\P(\mathscr H_j^{k})\ell_k^j=:\Sigma_{11}+\Sigma_{12}+\Sigma_{13}.\]

$\bullet$  {\it{Asymptotic behaviour of}} $\Sigma_{12}$\\
Applying Lemma \ref{recap}:
\begin{align*}
\Sigma_{12}&= \frac{2xy}{\pi}\sum_{k=\varepsilon x^2}^{x^2/\varepsilon} \sum_{j\in B_{k,\varepsilon}}\P(\mathscr H_j^{k})\frac{e^{-\frac{x^2}{2(k-j)}}}{(j+1)^{\frac{3}{2}}(k-j)^{\frac{3}{2}}}\left(1+O\left(\frac{y^2}{k}\right)+o\left(\frac{x^3}{k^2}\right)\right). 
\end{align*}
  Using formula \eqref{BE}, a lower bound for $\Sigma_{12}$ is given by :
\begin{align}
\Sigma_{12}&\ge
\frac{16xy}{\pi(1+\varepsilon)^3}\left(1+O\left(\frac{1}{x^{2\delta}}\right)\right)\sum_{k=\varepsilon x^2}^{x^2/\varepsilon}\frac{e^{-\frac{x^2}{k(1-\varepsilon)}}}{(k+1)^3}\sum_{j\in B_{k,\varepsilon}}\P(\mathscr H_j^{k})\nonumber\\
&\ge \frac{16xy}{\pi(1+\varepsilon)^3}\left(1+O\left(\frac{1}{x^{2\delta}}\right)\right)\sum_{k=\varepsilon x^2}^{x^2/\varepsilon}\frac{e^{-\frac{x^2}{k(1-\varepsilon)}}}{(k+1)^3}\left(1+O\left(\frac{1}{\sqrt{k}}\right)\right)\nonumber \\
&= \frac{16xy}{\pi(1+\varepsilon)^3}\left(1+O\left(\frac{1}{x}\right)+O\left(\frac{1}{x^{2\delta}}\right)\right)\sum_{k=\varepsilon x^2}^{x^2/\varepsilon}\frac{e^{-\frac{x^2}{k(1-\varepsilon)}}}{(k+1)^3}.\label{lbnd}
\end{align}
\noindent  
With the substitution $u=\nicefrac{x^2}{z}$: 
\begin{align*}
\sum_{k=\varepsilon x^2}^{{x^2}/{\varepsilon}}\frac{e^{-\frac{x^2}{k(1-\varepsilon)}}}{(k+1)^3}&\ge\left(1-\frac{2}{\varepsilon x^2+1} \right)^3\sum_{k=\varepsilon x^2}^{{x^2}/{\varepsilon}}\frac{e^{-\frac{x^2}{k(1-\varepsilon)}}}{(k-1)^3}
\ge \left(1-\frac{2}{\varepsilon x^2+1} \right)^3\sum_{k=\varepsilon x^2}^{{x^2}/{\varepsilon}}\int_{k-1}^k\frac{e^{-\frac{x^2}{z(1-\varepsilon)}}}{z^3} \mathrm{d}z\\
&\ge \left(1-\frac{2}{\varepsilon x^2+1} \right)^3\int_{\left\lfloor \varepsilon x^2\right\rfloor}^{\left\lfloor \frac{x^2}{\varepsilon}\right\rfloor}\frac{e^{-\frac{x^2}{z(1-\varepsilon)}}}{z^3} \mathrm{d}z
= \left(1-\frac{2}{\varepsilon x^2+1} \right)^3\frac{1}{x^4}\int_{\frac{x^2}{\left\lfloor \frac{x^2}{\varepsilon}\right\rfloor}}^{\frac{x^2}{\left\lfloor \varepsilon x^2\right\rfloor}}u e^{-\frac{u}{1-\varepsilon}}\mathrm{d}u\\
&=:g(x,\varepsilon). 
\end{align*}
By Lebesgue's dominated convergence theorem $\lim_{x\rightarrow +\infty}x^4g(x,\varepsilon)=\int_{\varepsilon }^{\frac{1}{\varepsilon}}u e^{-\frac{u}{1-\varepsilon}}\mathrm{d}u$ and $\lim_{\varepsilon \rightarrow 0}\lim_{x\rightarrow +\infty} $ $x^4g(x,\varepsilon)=\int_{0 }^{+ \infty}u e^{-\frac{u}{1-\varepsilon}}\mathrm{d}u=1$, this implies 
\[ \lim_{\varepsilon \rightarrow 0}\lim_{x\rightarrow +\infty} \frac{x^3}{y}\Sigma_{12} \geq  \frac{16}{\pi} \lim_{\varepsilon \rightarrow 0} \lim_{x\rightarrow +\infty}   \frac{1}{(1+\varepsilon)^3} \Big(1+O\left(\frac{1}{x}\right)+O\left(\frac{1}{x^{2\delta}}\right)\Big) \frac{x^4}{y}\sum_{\varepsilon x^2 \leq k \leq x^2/\varepsilon} g(x,\varepsilon)   \geq \frac{16}{\pi}. \]
Note that this last inequality implies the desired lower bound as:
\[\lim_{x\rightarrow+\infty}{\frac{x^3}{y}}\P_{(1,x)}(X_{\eta}=(y,0))\ge   
\lim_{\varepsilon \rightarrow 0}\lim_{x\rightarrow +\infty} \frac{x^3}{y}\Sigma_{12} \geq \frac{16}{\pi}.\] 
To obtain an upper bound, as  $\sum_{j\in B_{k,\varepsilon}}\P(\mathscr H_j^{k}) \leq 1$, 
\begin{align}
\Sigma_{12}&\le \frac{16xy}{\pi(1-\varepsilon)^3}\Big(1+O\Big(\frac{1}{x^{2\delta}}\Big)\Big)\sum_{k=\varepsilon x^2}^{x^2/\varepsilon}\frac{e^{-\frac{x^2}{(k+1)(1+\varepsilon)}}}{k^3}, \label{ubnd}
\end{align}
the rest of the proof is similar as the one for the lower bound implying that $\lim_{\varepsilon \rightarrow 0}\lim_{x\rightarrow +\infty} \frac{x^3}{y}\Sigma_{12}\le \frac{16}{\pi}$.
\noindent To complete the proof for the upper bound we have to show that $\Sigma_{11}$ and $\Sigma_{13}$ are negligible:\\
\noindent $\bullet$ \textit{$\Sigma_{13}$ negligibility:} \\
We just note, by \eqref{asympx} and \eqref{asympy2}, that
\begin{align}
\Sigma_{13} & \leq \frac{1}{4}\sum_{k  > x^2/\varepsilon} \max_{  j \in B_{k,\varepsilon} } \ell_k^j 
\leq  \frac{32}{\pi}\sum_{k  > x^2/\varepsilon} \frac{xy}{(k(1-\varepsilon))^3}e^{-\frac{x^2}{(k+1)(1+\varepsilon)}}\label{la39}\\
&\leq\frac{32xy}{\pi(1-\varepsilon)^3} \left(\frac{x^2-\varepsilon}{x^2(x^2+\varepsilon)}\right)^2\int_{0}^{\varepsilon\frac{x^2+\varepsilon}{x^2-\varepsilon}}ue^{-\frac{u}{1+\varepsilon}}\mathrm{d}u. \nonumber
\end{align}
This finally implies 
\[\lim_{\varepsilon\rightarrow0} \lim_{x\rightarrow+\infty}\frac{x^3}{y} \Sigma_{13} \le \lim_{\varepsilon\rightarrow0} \frac{32(1+\varepsilon)^2}{\pi(1-\varepsilon)^3}\int_0^{\frac{\varepsilon}{1+\varepsilon}}ve^{-v}\mathrm{d}v=0.\] 
$\bullet$ \textit{ $\Sigma_{11}$ negligibility} :\\ 
We use same first inequality as in \eqref{la39} and then split again the sum : let $0<\delta<1$
\begin{align*}
\Sigma_{11}& \leq \sum_{x \leq k \leq x^{2-\delta/2}}\max_{ j \in B_{k,\varepsilon} } \ell_k^j  +\sum_{x^{2-\delta/2} <  k \leq \varepsilon x^2}\max_{ j \in B_{k,\varepsilon} } \ell_k^j 
=: \Sigma_{11}^\prime+ \Sigma_{11}^*.
\end{align*}
{According to Lemma \ref{BE3}, for  $0<\varepsilon<\nicefrac{1}{3}$ and $x$ large enough, there exists $c_->0$ such that:
\begin{align*}
 \Sigma_{11}^\prime\le \sum_{x \leq k \leq x^{2-\delta/2}}\max_{ j \in B_{k,\varepsilon} } \ell_k^j \le \sum_{x \leq k \leq x^{2-\delta/2}}\max_{ j \in B_{k,\varepsilon} }\P(Z_{k-j-1} \geq x-1)\le \sum_{x \leq k \leq x^{2-\delta/2}}e^{-c_-\frac{x^2}{k}}\le x^{2-\delta/2}e^{-c_-{x^{\delta/2}}},
\end{align*}}
so $\lim_{x\rightarrow+\infty} \frac{x^3}{y} \Sigma_{11}^\prime=0 $.
And similarly as for $\Sigma_{13}$ above, using \eqref{asympx} and \eqref{asympy2}
\begin{align}
\Sigma_{11}^*& 
\leq \sum_{x^{2-\delta/2}<k  < \varepsilon x^2}\frac{32xy}{\pi(k(1-\varepsilon))^3}e^{-\frac{x^2}{(k+1)(1+\varepsilon)}}\leq \frac{32xy}{\pi(1-\varepsilon)^3} \left(\frac{x^2-\varepsilon}{x^2(x^2+\varepsilon)}\right)^2\int_{\frac{x^2+\varepsilon}{\varepsilon(x^2-\varepsilon)}}^{\infty}ue^{-\frac{u}{1+\varepsilon}}\mathrm{d}u.
\end{align}
Implying that : 
\[\lim_{\varepsilon \rightarrow 0} \lim_{x\rightarrow+\infty} \frac{x^3}{y}\Sigma_{11}^* \leq \lim_{\varepsilon \rightarrow 0}  \frac{C_+}{\varepsilon}e^{-\frac{1}{\varepsilon(1+\varepsilon)}}=0.\]
So finally $\lim_{\varepsilon \rightarrow 0} \lim_{x\rightarrow+\infty} \frac{x^3}{y}\Sigma_{11}=0$.\\
\noindent \underline{We now give some details for $\P_{(1,x)}(X_{\eta}=(0,y))$.} \\
As $y\leq x^{1- \delta}$, we can assume without loss of generality that $y<x$. Like the previous case:
\begin{align*}
\P_{(1,x)}(X_{\eta}=(0,y))&=\frac{1}{4}\sum_{k=\varepsilon x^2}^{x^2/\varepsilon}\sum_{j=0}^{k-1-(x-y)}\P(\mathscr H_j^{k})\P_{1}(\underline Z_{j}>0, Z_{j}=1)\P_x( \underline Z_{k-j-1}>0,  Z_{k-j-1}=y)\\
&=\frac{1}{4}\sum_{k\ge x}\sum_{j\in B_{k,\varepsilon}}\P(\mathscr H_j^{k})h_k^j+\frac{1}{4}\sum_{k\ge x}\sum_{j\in B^c_{k,\varepsilon}}\P(\mathscr H_j^{k})h_k^j=:\Omega_1+\Omega_2,
\end{align*}
where $B_{k,\varepsilon}^c$ is here the complementary of $B_{k,\varepsilon}$ in $\{0, \cdots, k-1-(x-y) \}$. According to \eqref{neglig1}, $\Omega_2$ is negligible.
Now, we only focus on the quantity $\Omega_{12}:=\frac{1}{4}\sum_{\varepsilon x^2\le k\le x^2/\varepsilon}\sum_{j\in B_{k,\varepsilon}}\P(\mathscr H_j^{k})h_k^j$ 
and \eqref{asympy} and \eqref{asympxy} give
 \begin{align*}
\Omega_{12}=\left(1+O\left(\frac{1}{x^{2 \delta}}\right)+o\left(\frac{1}{x}\right)\right){\frac{2xy}{\pi } }\sum_{k=\varepsilon x^2}^{x^2/\varepsilon} \sum_{j\in B_{k,\varepsilon}}\P(\mathscr H_j^{k})\frac{e^{-\frac{x^2}{2(k-j-1)}}}{(j+1)^{\frac{3}{2}}(k-j-1)^{\frac{3}{2}}}.
 \end{align*}
Using \eqref{BE}, with the same reasoning as the one for $(y,0)$: 
  \begin{align*}
 \Omega_{12}\ge \left(1+O\left(\frac{1}{x^{2 \delta}}\right)+o\left(\frac{1}{x}\right) \right)\frac{16xy}{\pi(1+\varepsilon)^3}\sum_{k=\varepsilon x^2}^{x^2/\varepsilon}\frac{e^{-\frac{x^2}{(k-1)(1-\varepsilon)}}}{(k+1)^{3}}\\
 \Omega_{12}\le \left(1+O\left(\frac{1}{x^{2 \delta}}\right)+o\left(\frac{1}{x}\right) \right)\frac{16xy}{\pi(1-\varepsilon)^3}\sum_{k=\varepsilon x^2}^{x^2/\varepsilon}\frac{e^{-\frac{x^2}{(k+1)(1+\varepsilon)}}}{(k-1)^{3}}.
 \end{align*}
 The upper and lower bound for $\Omega_{12}$ we obtain here is the equivalent of what we obtain for $\Sigma_{12}$ in the previous case (see \eqref{lbnd} and \eqref{ubnd}). Then the rest of the proof is very similar than what we previously detail 
so we chose to let this to the reader. The only difference is for the quantity $\sum_{x-y \leq k \leq x^{1+\delta}} \P(\mathscr H_j^k)h_k^j$, where we use again Lemma \ref{BE3}. \end{proof}

\begin{rem}\label{remeasy}
For any $\delta>0$, uniformly in $\oy \leq \ox^{1-\delta}$, there exists $C>0$ such that: 
\begin{align*}
    \P_{y}(\oX_\eta>\ox)\le C\frac{\oy}{\ox^2}.
\end{align*}
\end{rem}

\begin{lem}\label{simp3.1}
Assume $\alpha>3$. There exists $C>0$ such that for all $n\in\mathbb N^*$, all $1\le i\le n$ and $0<\varepsilon<\frac{\alpha-3}{4}$:
\begin{align*}
\P(\oX_{\eta_i}>n^{\nicefrac{1}{2}+2\varepsilon})\le \frac{C}{n^{1+4\varepsilon}}
\end{align*}
\end{lem}

\begin{proof}
For $i\ge2$, using the previous remark and Lemma \ref{lem410} (two times):
\begin{align*}
\P(\oX_{\eta_i}>n^{\nicefrac{1}{2}+2\varepsilon})&=\sum_{x\in K}\P(X_{\rho_{i-1}}=x)\P_x(\oX_{\eta}>n^{\nicefrac{1}{2}+2\varepsilon})\\
&\le \sum_{\ox\le n^{\nicefrac{1}{2}+\varepsilon}}\P(X_{\rho_{i-1}}=x)\P_x(\oX_{\eta}>n^{\nicefrac{1}{2}+2\varepsilon})+\sum_{\ox> n^{\nicefrac{1}{2}+\varepsilon}}\P(X_{\rho_{i-1}}=x)\\
&\le C \left(\frac{1}{n^{1+4\varepsilon}}\sum_{\ox\le n^{\nicefrac{1}{2}+\varepsilon}}\ox\P(X_{\rho_{i-1}}=x)
+\sum_{\ox> n^{\nicefrac{1}{2}+\varepsilon}}\frac{1}{\ox^{\alpha}}\right)
\le C\left(\frac{\E[\oX_{\rho_{i-1}}]}{n^{1+4\varepsilon}}+\frac{4}{n^{(\alpha-1)(\nicefrac{1}{2}+\varepsilon)}}\right)\\
&\le  C\left(\frac{1}{n^{1+4\varepsilon}}+\frac{1}{n^{(\alpha-1)(\nicefrac{1}{2}+\varepsilon)}}\right).
\end{align*}
As $0<\varepsilon<\frac{\alpha-3}{4}$, we have $1+4\varepsilon<(\alpha-1)(\nicefrac{1}{2}+\varepsilon)$. Note that for $i=1$, the proof is easier using straightly Remark \ref{remeasy} and we have our claimed result. 
\end{proof}

\begin{lem} \label{MarcheSimpleTemps} For any $y\in \partial K$ such that $\oy=o(k^{1/2})$, $\lim_{k \rightarrow +\infty} \frac{k^2}{\oy}\P_{y}(\eta=k)= \frac{8}{\pi}$.
\end{lem}
\begin{proof}
 For typographical simplicity, we only treat $\P_{(1,x)}(\eta>n)$ with $x>0$ (the other ones can be obtain by symmetry). Recall the definition of $\Hm_.^{.}$ before \eqref{lkj},  following the same ideas as for the proof of Lemma \ref{MarcheSimple}, for $\varepsilon>0$:
\begin{align*}
\P_{(1,x)}({\eta}=k)  &= \sum_{m\le k-1}  \P_{(1,x)}( \eta=k|\mathscr H_m^{k} )\P(\mathscr H_m^{k} )\\
&= \Big(\sum_{m\in B_{k,\varepsilon}}+ \sum_{m\in B_{k,\varepsilon}^c}\Big) \P_{(1,x)}( \eta=k|\mathscr H_m^{k} )\P(\mathscr H_m^{k} ):=\Sigma_1+\Sigma_2,
\end{align*}
and we can prove that $\Sigma_2$ is negligible (see \eqref{neglig1} for details). Thus we only study $\Sigma_1$ and write:
\begin{align*}
\Sigma_1
&= \frac{1}{4}    \sum_{m\in B_{k,\varepsilon}}\P(\Hm_m^{k})  \P_1(\underline{Z}_m >0) \P_{x}(\underline{Z}_{k-m-1} >0, Z_{k-m-1}=1) \\
& + \frac{1}{4}  \sum_{m\in B_{k,\varepsilon}} \P(\Hm_m^{k})   \P_{1}(\underline{Z}_{k-m-1} >0, Z_{k-m-1}=1) \P_{x}(\underline{Z}_m >0)=:\Sigma_{11}+\Sigma_{12}.
\end{align*}
For any large $k$ and any $x=o(k^{1/2})$, with Corollary \ref{corsup}, \eqref{asympx} and \eqref{BE}:
\begin{align*}
 \Sigma_{11}&\ge\frac{x}{\pi}  \sum_{ m\in B_{k,\varepsilon} } \P(\Hm_m^{k}) \frac{e^{-\frac{x^2}{2(k-m)}}}{\sqrt{ m}(k-m)^{\frac{3}{2}}}\Big(1+ o\left(\frac{x^3}{(k-m)^2}\right) +o\left(\frac{1}{m}\right)\Big)\\
& \geq \frac{4x}{\pi(1+ \varepsilon)^2} \frac{e^{-\frac{x^2}{k(1- \varepsilon)}}}{k^{2}}\Big(1+ o\left(\frac{x^3}{k^2}\right)+o\left(\frac{1}{k}\right) \Big) \sum_{ m \in B_{k,\varepsilon}}  \P(\Hm_m^{k}) \\ 
&\geq 
  (1-2\varepsilon) \frac{ 4x}{{\pi}}    \frac{e^{-\frac{x^2}{k(1- \varepsilon)}}}{k^{2}} \geq (1-3\varepsilon) \frac{4 x}{{\pi} k^2}.
\end{align*}
The upper-bound is simpler as: 
\begin{align*}
\Sigma_{11}\le  \frac{4x}{\pi(1- \varepsilon)^2} \frac{1}{(k-1)^{2}}\Big(1+ o\left(\frac{x^3}{k^2}\right)+o\left(\frac{1}{k}\right) \Big) \le \frac{4x}{\pi k^2}(1+2\varepsilon) 
\end{align*}
Thus, $\lim_{\varepsilon\rightarrow0}\lim_{n\rightarrow\infty}n\Sigma_{11}/x=4/\pi$. $\Sigma_{12}$ can be treated similarly and we obtain our claimed result.
\end{proof}

We conclude this section with two lemmata, the first one is the counterpart on $K$ of \red{Lemma \ref{lem410}} and the last one a useful identity:
\begin{lem}\label{newlem51}
For any ${0<\beta <2}$ there exists $C>0$ such that  for all $i\ge 1$:
\begin{equation*}
    \E\left[\overline{ X}_{\eta_i}^{\beta}\right]\le C.
\end{equation*}
\end{lem}

\begin{proof}
According to \cite{McConnell}, Theorem 1.3 page 223 (see also \cite{Denwac} Lemma 10 page 1007)
, for any $0<\beta <2$ there exists $C>0$ such that for any $x\in K$: 

\begin{equation} 
\E_x\left[\max_{i \leq \eta} \overline X_i^{\beta}\right]
\leq C (1+ \overline x^{\beta}).
\end{equation}
 Using the strong Markov Property (two times), previous inequality and Corollary \ref{mean} yields
 \begin{align}
\E\left[\overline{ X}_{\eta_i}^{\beta}\right]&= \E\left[\E_{X_{\rho_{i-1}}}\left[\overline{X }_{\eta}^{\beta}\right]\right]\le C\E\left[\overline{X }_{\rho_{i-1}}^{\beta}\right]=C\E\left[\E_{X_{\eta_{i-1}}}\left[\overline{X }_{\rho}^{\beta}\right]\right]\le C\E[M]=CM.\end{align}
\end{proof}

\begin{lem}\label{reverse}
For all $x\in K^c$ :
\begin{align*}
\sum_{y\in\partial K}\P_y\left(X_\eta=x\right)= 2.
\end{align*}
\end{lem}
\begin{proof}
We can assume without loss of generality that $x=(0,i)$ for $i\ge 1$. Using the reversibility: 
\begin{align*}
\sum_{y\partial K}\P_y\left(X_\eta=(0,i)\right)&=\P_{(1,1)}\left(X_\eta=(0,i)\right)+\sum_{j\ge 2}\P_{(j,1)}\left(X_\eta=(0,i)\right)+\P_{(1,j)}\left(X_\eta=(0,i)\right)\\
&+\P_{(-1,1)}\left(X_\eta=(0,i)\right)+\sum_{j\ge 2}\P_{(-1,j)}\left(X_\eta=(0,i)\right)+\P_{(-j,1)}\left(X_\eta=(0,i)\right)\\
&=\P_{(1,i)}\left(X_\eta=(0,1)\right)+\sum_{j\ge 2}\P_{(1,i)}\left(X_\eta=(j,0)\right)+\P_{(1,i)}\left(X_\eta=(0,j)\right)\\
&+\P_{(-1,i)}\left(X_\eta=(-1,0)\right)+\sum_{j\ge 2}\P_{(-1,i)}\left(X_\eta=(0,j)\right)+\P_{(-1,i)}\left(X_\eta=(-j,0)\right)=2.
\end{align*}
\end{proof}

\section{Appendix \label{sec6}}

In this appendix, we give asymptotic results linked to $(Z_n)_{n\ge 0}$, the symmetric random walk on $\mathbb Z$, results that we used throughout this paper. Recall that  $B_{k,\varepsilon}=\left[\nicefrac{(k-1)(1-\varepsilon)}{2},\nicefrac{(k-1)(1+\varepsilon)}{2}\right]$ and $\mathscr H_j^{k}$ is the event  {\it{$\lbrace $among the first $k-1$ steps, there is exactly $j$ horizontal ones$\rbrace$} .}

\begin{lem}\label{ineq} 
Let $0<\delta<1$, assume that $k$ is an integer such that $ \lim_{x \rightarrow + \infty} \nicefrac{\ln k}{ \ln x} \in [(2- \delta),2]$, then
\begin{equation}\label{ineqx}
2^{-k}\binom{k}{\frac{k-x}{2}}={\sqrt{\frac{2}{\pi k} }}e^{-\frac{x^2}{2k}}\left(1+o\left(\frac{x^3}{k^2}\right)\right).
\end{equation}
If, moreover there exists $y\le x^{1-\delta}$:
\begin{equation}
2^{-k}\binom{k}{\frac{k-y}{2}}={\sqrt{\frac{2}{\pi k} }}\left(1+O\left(\frac{y^2}{k}\right)\right).
\end{equation}
\end{lem}

\begin{proof}
Using Stirling formula: 
\begin{align*}
2^{-k}\binom{k}{\frac{k-x}{2}}&=\frac{2^{-k}}{\sqrt{2\pi }}\frac{k^{k+\frac{1}{2}}e^{-k}\left(1+\frac{1}{12k}+o\left(\frac{1}{k}\right)\right)}{\left(\frac{k-x}{2}\right)^{\frac{k-x+1}{2}}e^{-\frac{k-x}{2}}\left(\frac{k+x}{2}\right)^{\frac{k+x+1}{2}}e^{-\frac{k+x}{2}}\left(1+\frac{1}{6k}+o\left(\frac{1}{k}\right)\right)}\\
&={\sqrt{\frac{2}{\pi k} }}\frac{1}{\left(1-\frac{x}{k}\right)^{\frac{k-x+1}{2}}\left(1+\frac{x}{k}\right)^{\frac{k+x+1}{2}}}\left(1-\frac{1}{12k}+o\left(\frac{1}{k}\right)\right).\\
\end{align*}
Moreover
\begin{align*}
A&:=\left(1-\frac{x}{k}\right)^{\frac{k-x+1}{2}}\left(1+\frac{x}{k}\right)^{\frac{k+x+1}{2}}
=e^{\frac{k-x+1}{2}\ln \left(1-\frac{x}{k}\right)}e^{\frac{k+x+1}{2}\ln \left(1+\frac{x}{k}\right)}\\
&=\exp\left(\frac{k-x+1}{2}\left(-\frac{x}{k}-\frac{x^2}{2k^2}-\frac{x^3}{3k^3}+o\left(\frac{x^3}{k^3}\right)\right)  
+\frac{k+x+1}{2}\left(\frac{x}{k}-\frac{x^2}{2k^2}+\frac{x^3}{3k^3}+o\left(\frac{x^3}{k^3}\right)\right)\right)\\
&=\exp\left(x\left(\frac{x}{k}+\frac{x^3}{3k^3}\right)-\frac{x^2}{2k^2}\left(k+1\right)+o\left(\frac{x^3}{k^2}\right)\right)=\exp\left(\frac{x^2}{2k}+\frac{x^4}{3k^3}-\frac{x^2}{2k^2}+o\left(\frac{x^3}{k^2}\right)\right)\\
&=\exp\left(\frac{x^2}{2k}+o\left(\frac{x^3}{k^2}\right)\right)=\exp\left(\frac{x^2}{2k}\right)\left(1+o\left(\frac{x^3}{k^2}\right)\right).
\end{align*}
In the first case: 
\begin{align*}
2^{-k}\binom{k}{\frac{k-x}{2}}&={\sqrt{\frac{2}{\pi k} }}{e^{-\frac{x^2}{2k}}\left(1+o\left(\frac{x^3}{k^2}\right)\right)}\left(1-\frac{1}{12k}+o\left(\frac{1}{k}\right)\right)={\sqrt{\frac{2}{\pi k} }}{e^{-\frac{x^2}{2k}}\left(1+o\left(\frac{x^3}{k^2}\right)   \right)},
\end{align*}
If $y\le x^{1-\delta}$, $e^{-\frac{y^2}{2k}}=1-\frac{y^2}{2k}+o\left(\frac{y^2}{2k}\right)=1+O\left(\frac{y^2}{k}\right)$, 
giving the second formula.
\end{proof}

\begin{cor}\label{corsup}
When $k$ goes to infinity, for any $m\in B_{k,\varepsilon}$ and any $u=o(k^{1/2})$, 
\begin{align*}
u\sqrt{\frac{2}{\pi m}}\left(1+o\left(\frac{u}{m}\right)\right)\le  \P_{u}\left(\underline{Z}_m >0\right)\le u\sqrt{\frac{2}{\pi m}}\left(1+o\left(\frac{1}{m}\right)\right).
\end{align*}
\end{cor}
\begin{proof}
According to \cite{Feller}  pp.72 and pp.88-89, $\P_{u}\left(\underline{Z}_m =j\right)=2^{-m}\binom{m}{\frac{m-(j-u)}{2}}\vee\binom{m}{\frac{m-(j-u)-1}{2}}$, 
then, assuming $m \in B_{k,\varepsilon}$: 
\begin{align*}
u\P_{u}\left(\underline{Z}_m =1\right)\le \P_{u}\left(\underline{Z}_m >0\right)=\sum_{j=1}^u\P_{u}\left(\underline{Z}_m =j\right)\le u\P_{u}\left(\underline{Z}_m =u\right).
\end{align*}
As $u=o(k^{1/2})$, we conclude using Lemma \ref{ineq},
\end{proof}

\begin{lem}\label{recap}
When $k$ goes to infinity
\begin{align}\label{asympy}
\P_{1}(\underline Z_{k}>0, Z_{k}=1)&=\sqrt{\frac{2}{\pi}}\frac{2}{(k+1)^{\frac{3}{2}}}\left(1+o\left(\frac{1}{k^2}\right)\right). 
\end{align}
Let $0<\delta<1$, for any $x$ large enough and for any $k \geq x^{2- \delta/2}$ ,
\begin{align}\label{asympx}
\P_x\left(\underline Z_{k}>0, Z_{k}=1\right)=\sqrt{\frac{2}{\pi}}\frac{2x}{(k+1)^{\frac{3}{2}}}e^{-\frac{x^2}{2(k+1)}}\left(1+o\left(\frac{x^3}{k^2}\right)\right).
\end{align}
If, moreover $y\le x^{1-\delta}$,
\begin{align}
\P_{1}(\underline Z_{k}>0, Z_{k}=y)&=\sqrt{\frac{2}{\pi}}\frac{2y}{(k+1)^{\frac{3}{2}}}\left(1+O\left(\frac{y^2}{k}\right)+o\left(\frac{1}{k^2}\right)\right),\label{asympy2}\\
\P_x\left(\underline Z_{k}>0, Z_{k}=y\right)&=\sqrt{\frac{2}{\pi}}\frac{2xy}{k^{\frac{3}{2}}}e^{-\frac{x^2}{2k}} \left(1+o\left(\frac{x^3}{k^2}\right)+O\left(\frac{y^2}{k}\right) + O \left(\frac{(xy)^2}{k^2}\right)  \right). \label{asympxy}
\end{align}
\end{lem}

\begin{proof}
Using the stationarity of $(Z_k)_{k\ge0}$ and the Desire Andr\'e's reflexion principle (see  \cite{Feller}, p. 72-73 and 95 problems for solution):
\begin{align*}
\P_x( \underline Z_{k}>0,  Z_{k}=1)=\P( \underline Z_{k}>-x,  Z_{k}=1-x)=\frac{1}{2^k}\left(\binom{k}{\frac{k+1-x}{2}}-\binom{k}{\frac{k-x-1}{2}}\right)=\left(\frac{1}{2}\right)^{k+1}\frac{2x}{k+1}\binom{k+1}{\frac{k+1-x}{2}},
\end{align*}
and we easily obtain \eqref{asympy} taking $x=1$ and using Lemma \ref{ineq}. The reasonings to obtain \eqref{asympx} and \eqref{asympy2} are very similar and use again Lemma \ref{ineq}.\\ 
To prove \eqref{asympxy}, we use again the Desire Andr\'e's reflexion principle and formula \eqref{ineqx}
\begin{align*}
\P_x( \underline Z_{k}>0,  Z_{k}=y)&=\P( \underline Z_{k}>-x,  Z_{k}=y-x)
=\left(\frac{1}{2}\right)^{k}\left(\binom{k}{\frac{k-(x-y)}{2}}-\binom{k}{\frac{k-(x+y)}{2}}\right)\\
&={\sqrt{\frac{2}{\pi k} }}\left(1+o\left(\frac{x^3}{k^2}\right)\right)\left(e^{-\frac{(x-y)^2}{2k}}-e^{-\frac{(x+y)^2}{2k}}\right)  \\
&={\sqrt{\frac{2}{\pi k} }}\left(1+o\left(\frac{x^3}{k^2}\right)\right)e^{-\frac{x^2+y^2}{2k}}\left(e^{\frac{xy}{k}}-e^{-\frac{xy}{k}}\right)\\
&={\sqrt{\frac{2}{\pi k} }}\left(1+o\left(\frac{x^3}{k^2}\right)\right)e^{-\frac{x^2+y^2}{2k}}\left(\frac{2xy}{k}+O \left(\frac{(xy)^3}{k^3}\right)\right)\\
&={\sqrt{\frac{2}{\pi} }}\frac{2xy}{k^{\frac{3}{2}}}e^{-\frac{x^2+y^2}{2k}}\left(1+o\left(\frac{x^3}{k^2}\right)\right)\left(1+O \left(\frac{(xy)^2}{k^2}\right)\right) \\
& =\sqrt{\frac{2}{\pi} }\frac{2xy}{k^{\frac{3}{2}}}e^{-\frac{x^2}{2k}}\left(1+o\left(\frac{x^3}{k^2}\right)\right)\left(1+O\left(\frac{y^2}{k}\right)\right)\left(1+O \left(\frac{(xy)^2}{k^2}\right) \right). 
\end{align*}\end{proof}

\noindent To finish we recall elementary facts use several times in Section  \ref{lecone}.
\begin{lem} 
When $k$ goes to infinity:
\begin{equation}\label{BE}
  \sum_{j\in B_{k,\varepsilon}}\P(\mathscr H_j^{k}) \geq 1-O\Big(\frac{1}{\sqrt{k}}\Big),
\end{equation}
and for all  $j\notin B_{k,\varepsilon}$
\begin{equation}\label{BE2}
 \P(\mathscr H_j^{k}) \le  e^{-\frac{\varepsilon^2 k}{6}}. 
\end{equation}
\end{lem}

\begin{proof}
Formulas \eqref{BE} and \eqref{BE2} are respectively applications of the Berry-Esseen and Chernoff inequalities.
\end{proof}

\begin{lem}\label{BE3}
For $k>x-y$ with $y=o(x)$ and $0<\varepsilon<\nicefrac{1}{3}$, there exists $c_->0$ such that for $x$ large enough:
\begin{align*}
\max_{j\in B_{k,\varepsilon}}\P(Z_{j} \geq x-y)  \leq e^{-\frac{c_- x^{2}}{k}}.
\end{align*}
\end{lem}

\begin{proof} Proof is elementary, we give some details for completeness. Assume $y=1$ for the moment, for any $\alpha>0$: 
\begin{align*}
\P(Z_{j} \geq x-1)&=\P\left(e^{\alpha Z_{j}} \geq e^{\alpha(x-1)}\right)\le \E\left[e^{\alpha Z_{j}} \right]e^{-\alpha(x-1)}=\E\left[e^{\alpha Z_{1}} \right]^{j}e^{-\alpha(x-1)}\\
&=e^{j\ln \E\left[e^{\alpha Z_{1}} \right]-\alpha(x-1)}\le e^{j\left(\E\left[e^{\alpha Z_{1}} \right]-1\right)-\alpha(x-1)}=e^{j\left(\cosh \alpha-1\right)-\alpha(x-1)}.
\end{align*}
For $0<\alpha<2$, $\cosh \alpha\le1+ \frac{3}{4}\alpha^2$ implying that for such $\alpha$:
\begin{align*}
\P(Z_{j} \geq x-1)&\le e^{\frac{3}{4}j\alpha^2-\alpha(x-1)}=:e^{\psi(\alpha)}.
\end{align*}
One can find that $\psi$ reaches its minimum for $\alpha=\frac{2(x-1)}{3j}$ (note that for $0<\varepsilon<\nicefrac{1}{3}$ and $k>x$, this quantity is bounded above by 2) and for that choice of $\alpha$:
\begin{align*}
\P(Z_{j} \geq x-1)& \le e^{-\frac{(x-1)^2}{3j}}\le e^{-\frac{(x-1)^2}{3k}}\le e^{-\frac{x^2}{6k}},
\end{align*} 
for $x$ large enough. This finishes the proof for $y=1$, for general $y$ proof is identical just the optimal $\alpha$ changes in $\alpha={2(x-y)}/{3j}$.
\end{proof}

\noindent \textbf{Acknowledgment: }We would like to thank Julien Barr\'e for introducing us the work of physicists  and the problematics which inspire this paper.

\bibliographystyle{alpha}
\bibliography{thbiblio}

\begin{thebibliography}{McC84}

\bibitem[Dal88]{Dalibard}
J.~Dalibard.
\newblock Laser cooling of an optically thick gas: the simplest radiation
  pressure trap?
\newblock {\em Optics Communications}, 68:203--208, 1988.

\bibitem[DW15]{Denwac}
Denis Denisov and Vitali Wachtel.
\newblock {Random walks in cones}.
\newblock {\em The Annals of Probability}, 43(3):992 -- 1044, 2015.

\bibitem[ET60]{Erdor}
P.~Erd\H{o}s and S.J. Taylor.
\newblock {Some problems concerning the structure of random walk paths}.
\newblock {\em Acta Math. Acad. Sci . Hung.}, 11(3):137 -- 162, 1960.

\bibitem[Fel68]{Feller}
W.~Feller.
\newblock {\em An Introduction to Probability Theory, Vol. 1}.
\newblock Wiley, New York, NY, third edition, 1968.

\bibitem[McC84]{McConnell}
Terry~R. McConnell.
\newblock Exit times of n-dimensional random walks.
\newblock {\em Z. Wahrscheinlichkeitstheorie verw. Gebiete}, 67(67):\ 213--233,
  1984.

\bibitem[R\'89]{Revesz}
P.~R\'ev\'esz.
\newblock {\em Random walk in random and non-random environments}.
\newblock World Scientific, 1989.

\bibitem[Ras09]{Raschel}
K.~Raschel.
\newblock Random walks in the quarter plane absorbed at the boundary : Exact
  and asymptotic.
\newblock {\em arxiv : https://arxiv.org/abs/0902.2785}, 2009.

\end{thebibliography}

\end{document}